\documentclass{amsart}

\usepackage{graphicx}
\usepackage{hyperref}
\usepackage{cite}
\usepackage[utf8]{inputenc}
\usepackage[english]{babel}
\usepackage{amsmath}
\usepackage{amsfonts}
\usepackage{amssymb}
\usepackage{xcolor}
\usepackage{amsthm}
\usepackage[all]{xy}
\usepackage{enumitem}

\newtheorem{thm}{Theorem}[section]
\newtheorem{lemma}[thm]{Lemma}
\newtheorem{prop}[thm]{Proposition}
\newtheorem{cor}[thm]{Corollary}

\theoremstyle{definition}
\newtheorem{defn}[thm]{Definition}

\theoremstyle{remark}
\newtheorem{rem}[thm]{Remark}

\numberwithin{equation}{section}

\DeclareMathOperator{\Ann}{Ann}

\DeclareMathOperator{\add}{add}
\DeclareMathOperator{\Ker}{Ker}
\DeclareMathOperator{\Img}{Im}
\DeclareMathOperator{\Mat}{Mat}
\DeclareMathOperator{\Hom}{Hom}
\DeclareMathOperator{\End}{End}

\DeclareMathOperator{\Sat}{Sat}

\newcommand{\sm}{\sigma[M]}

\newcommand{\ess}{\leq^\text{ess}}

\begin{document}
	
	\title{Reduced rank in $\sm$}
	
	\author{John A. Beachy}
	\address{Department of Mathematical Sciences, Northern Illinois University,
		DeKalb 60115, Illinois, US}
	\email{jbeachy@niu.edu}
	
	\author{Mauricio Medina-B\'arcenas}
	\address{Departamento de Matem\'{a}ticas, 
		Universidad Autónoma Metropolitana-Iztapalapa,
		Sn. Rafael Atlixco 186, Col. Vicentina 09340
		Ciudad de México, M\'{e}xico}
	\email{mmedina@xanum.uam.mx}
	\thanks{The second author was supported by the grant ``CONACYT-Estancias Posdoctorales 2do A\~no 2020 - 1''.}
	
	\subjclass[2010]{Primary 16D90, 16P50; Secondary 16P70, 16S50}
	
	
	\date{}
	
	
	\keywords{Finite reduced rank, Goldie module, Quotient category, Hereditary torsion theory}
	
	\begin{abstract}
		Using the concept of prime submodule 
		introduced by Raggi et.al.\ 
		we extend the notion of reduced rank to the module-theoretic context of $\sm$. 
		We study the quotient category of $\sm$ 
		modulo the hereditary torsion theory 
		cogenerated by the $M$-injective hull of $M$,
		when $M$ is a semiprime Goldie module. 
		We prove that this quotient category is spectral. 
		We then consider the hereditary torsion theory in $\sm$ 
		cogenerated by the $M$-injective hull of $M/\mathfrak{L}(M)$, 
		where $\mathfrak{L}(M)$ is the prime radical of $M$, 
		and we determine when the module of quotients of $M$, 
		with respect to this torsion theory, 
		has finite length in the quotient category. 
		Finally, we give conditions on a module $M$ with endomorphism ring $S$ 
		under which $S$ is an order in an Artinian ring, 
		extending Small's Theorem.
	\end{abstract}
	
	\maketitle
	
	\section{Introduction}
	
	The notion of reduced rank 
	was introduced by Goldie in \cite{goldie1964torsion}. 
	Over a semiprime left Goldie ring $R$ with classical ring of quotients $Q$, 
	the reduced rank of an $R$-module $M$ is defined 
	as the dimension of the semisimple $Q$-module $Q\otimes_R M$. 
	This definition was extended to finitely generated modules 
	over any left Noetherian ring,
	using the fact that the prime radical of a left Noetherian ring is nilpotent. 
	The reduced rank of an $R$-module $M$ over a left Noetherian ring $R$ 
	has been shown to be a very useful notion 
	(see, for example, \cite{chatters1979reduced}, 
	\cite[Chapter~12]{lamlectures}, and 
	\cite[Chapter~4]{mcconnell2001noncommutative}).  
	In particular, 
	it has been used to characterize left orders in left Artinian rings. 
	
	The first author extended the notion of reduced rank 
	to arbitrary rings in \cite{beachy1982rings}. 
	Let $\gamma$ be the hereditary torsion theory 
	cogenerated by the injective hull of $R/N$, 
	where $N$ is the prime radical of the ring $R$. 
	Then the ring $R$ is said to have finite reduced rank 
	if the quotient ring $\mathcal{Q}_\gamma(R)$ 
	has finite length in the quotient category $R$-Mod/$\gamma$. 
	This definition agrees with that given by Goldie 
	when $R$ is a left Noetherian ring, 
	and it has been shown 
	(see \cite{lenagan} and \cite{beachy1982rings}) 
	that a ring with Krull dimension has finite reduced rank
	in this more general sense. 
	If the ring $R$ has finite reduced rank, 
	then $R/N$ is a left Goldie ring
	\cite[Theorem 1]{beachy1982rings}. 
	The general notion of reduced rank can be used 
	to characterize left orders in left Artinian rings,
	extending Small's Theorem \cite[Theorem 4]{beachy1982rings} 
	(for Small's Theorem itself see \cite[4.1.4]{mcconnell2001noncommutative}). 
	
	In this paper we are interested 
	in a more general approach to the concept of reduced rank. 
	A new notion of a prime submodule of a given module
	was introduced in \cite{raggiprime},
	using the product of submodules defined in \cite{BicanPr}.
	It is then natural to consider the prime radical of a module $M$: 
	the intersection of all prime submodules of $M$,
	which we will denote by $\mathfrak{L}(M)$. 
	The appropriate context in which to work appears to be
	the category $\sm$ of all modules subgenerated by $M$.
	We will say that a module $M$ has finite reduced rank in $\sm$ 
	if the module of quotients $\mathcal{Q}_\gamma(M)$ 
	has finite length in the quotient category $\sm/\gamma$, 
	where $\gamma$ is the hereditary torsion theory in $\sm$ 
	cogenerated by the $M$-injective hull of $M/\mathfrak{L}(M)$.
	As in the case of rings, 
	we show if $M$ has Krull dimension and is a projective generator of $\sm$,
	then $M$ has finite reduced rank in $\sm$. 
	This provides an interesting class of examples of the general theory.
	
	
	The remainder of this paper is divided into sections as follows. 
	In Section \ref{httsm}, 
	we present some properties of hereditary torsion theories 
	in the category $\sm$ that are needed in the following sections. 
	
	In Section \ref{qcsgm}, we study the quotient category $\sm/\tau_g$, 
	where $M$ is a semiprime Goldie module 
	and $\tau_g$ is the hereditary torsion theory 
	cogenerated by the $M$-injective hull of $M$ in $\sm$. 
	Under the hypothesis that $M$ is projective in $\sm$,
	we prove in Theorem \ref{goldiefl}  
	that for a Goldie module $M$ 
	with $M$-injective hull $\widehat{M}$, 
	the quotient category $\sm/\tau_g$ is discrete spectral 
	and equivalent to Mod-$\End_R(\widehat{M})$ and
	that the module of quotients 
	$\mathcal{Q}_{\tau_g}(M)$ is semisimple of finite length. As Corollary, we characterize a semiprime left Goldie ring $R$ in terms of its maximal left rings of quotients $Q_{max}^\ell(R)$ (Corollary \ref{qmax}).  
	
	In Section \ref{mfrr}, 
	we introduce the main concept of the paper,
	that of reduced rank in $\sm$. 
	We characterize a module $M$ with finite reduced rank in $\sm$ 
	in terms of its lattice of $\gamma$-saturated submodules (Theorem \ref{rrt}), 
	and as a consequence $M/\mathfrak{L}(M)$ is a Goldie module. 
	We determine when the property of having finite reduced rank 
	is inherited by submodules, 
	factor modules, and direct sums 
	(Proposition \ref{rreq} and Proposition \ref{rrs}). 
	We prove that if $M$ has finite reduced rank in $\sm$, 
	then so does every generator of $\sm$ (Proposition \ref{rrpg}). 
	This allows us to prove that 
	having finite reduced rank is preserved 
	under equivalences between categories of the type $\sm$,
	and to show that having finite reduced rank 
	is a Morita invariant property 
	(Proposition \ref{rreq} and Corollary \ref{rrMorita}). 
	Finally, we show that for a semicentral idempotent $e^2=e$ in a ring $R$ 
	with finite reduced rank, 
	the corner ring $eRe$ 
	inherits the property (Corollary \ref{corner}). 
	
	In Section \ref{gST} we generalize Small's Theorem\footnote{and the more general version given by Warfield in \cite{warfield1979bezout}}.
	Given a module $M$ with endomorphism ring $S$ 
	and a fully invariant submodule $N\leq M$, 
	we consider the set $\mathcal{C}(N)$ 
	of those endomorphisms $f\in S$ 
	which induce a monomorphism $\overline{f}:M/N\to M/N$. 
	We prove that for a projective generator $M$ in $\sm$, 
	the module $M$ has finite reduced rank in $\sm$ 
	and $\mathcal{C}(\mathfrak{L}(M))\subseteq\mathcal{C}(0)$ 
	if and only if 
	$T=\End_R(\mathcal{Q}_{\gamma}(M))$ is an Artinian ring 
	and $S$ is a right order in $T$ (Theorem \ref{smallgen}). To finish, we apply the last theorem to corner rings and Rickart modules.
	
	Throughout this paper $R$ will be an associative ring with unity, 
	and all $R$-modules will be unitary left $R$-modules. 
	The category of $R$-modules is denoted as $R$-Mod. 
	Given an $R$-module $M$, 
	the full subcategory of $R$-Mod 
	consisting of all those modules 
	which can be embedded in an $M$-generated module is called $\sm$. 
	As a class of modules, 
	$\sm$ is a hereditary pretorsion class, 
	that is, $\sm$ is closed under submodules, 
	direct sums and homomorphic images. 
	Moreover, $\sm$ has direct products and injective hulls, 
	although these constructions differ from those in $R$-Mod. 
	It can be seen that $\sm$ is a Grothendieck category,
	and so it has  a generator $U$ 
	(which may be not $M$). 
	Given a family $\{N_i\}_I$ of modules in $\sm$, 
	the direct product in $\sm$ of the family $\{N_i\}_I$ 
	can be constructed as the trace $tr^{U}(\prod_IN_i)$ 
	of $U$ in the direct product in $R$-Mod of the family $\{N_i\}_I$. 
	We let $\prod_I^{[M]}N_i$ denote the direct product in $\sm$ 
	of the family $\{N_i\}_I$. 
	Every module in $\sm$ has an injective hull in $\sm$. 
	To show this, 
	let $N$ be any module in $\sm$, 
	and let $E(N)$ be its injective hull in $R$-Mod. 
	Then $tr^M(E(N))\neq 0$ and $N\subseteq tr^M(E(N))$, 
	and it can be seen that $tr^M(E(N))$ is an injective hull in $\sm$ for $N$. 
	We let $E^{[M]}(N)$ denote the injective hull 
	(or $M$-injective hull) of $N$ in $\sm$. 
	Note that $E^{[M]}(N)$ is always an $M$-generated module. 
	For general information on rings, modules and the category $\sm$, 
	the reader is referred to 
	\cite{lamlectures}, \cite{stenstromrings}, and \cite{wisbauerfoundations}.
	
	\section{Hereditary Torsion Theories in $\sigma [M]$}
	\label{httsm}
	
	Following Wisbauer \cite[9.5]{wisbauermodules}, 
	given a hereditary torsion theory $\tau=(\mathcal{T},\mathcal{F})$ in $\sm$, 
	there exists an $M$-injective module $E$ such that 
	\[ \mathcal{T}=\{N\in\sm\mid \Hom_R(N,E)=0\} . \] 
	We claim that 
	\[ \textstyle \mathcal{F}=\{L\in\sm\mid L\hookrightarrow \prod_I^{[M]}E\} , \] 
	where  $\prod_I^{[M]}E$ is the direct product of copies of $E$ 
	in the category $\sm$. 
	To show this, let $N$ in $\mathcal{T}$ and let
	\[ \textstyle L\in\{L\in\sm\mid L\hookrightarrow \prod_I^{[M]}E\} . \] 
	Then there is a monomorphism $\alpha:L\hookrightarrow \prod_I^{[M]}E$ 
	for some index set $I$. 
	Suppose that $f:N\to L$ is a nonzero homomorphism. 
	It follows that there is an index $i\in I$ 
	such that $\pi_i\alpha f\neq 0$, 
	where $\pi_i:\prod_I^{[M]}E\to E$ is the canonical projection. 
	This is a contradiction because $N$ is in $\mathcal{T}$. 
	On the other hand, let $L$ be in $\mathcal{F}$. 
	Then $\Hom_R(L,E)\neq 0$, 
	and we consider the canonical homomorphism 
	$\alpha:L\to \prod_{\Hom_R(L,E)}^{[M]}E$. 
	It follows that $\Ker\alpha=\bigcap_{f\in\Hom_R(L,E)}\Ker f$. 
	Since $E$ is injective in $\sm$, 
	we have $\Hom_R(\Ker\alpha,E)=0$. 
	This implies that $\Ker\alpha\in\mathcal{T}\cap\mathcal{F}=0$,
	and thus $\alpha$ is a monomorphism. 
	
	It follows that a hereditary torsion theory 
	$\tau=(\mathcal{T},\mathcal{F})$ in $\sm$ 
	is cogenerated by an $M$-injective module $E$, that is,
	\begin{equation}\label{11}
		\begin{split}
			\mathcal{T}& =\{N\in\sm\mid \Hom_R(N,E)=0\} \\
			\mathcal{F}& =\{L\in\sm\mid L\hookrightarrow \textstyle{\prod_I^{[M]}E}\;
			\text{for some index set}\;I\}
		\end{split}
	\end{equation}
	
	\begin{lemma}
		Let $M$ be an $R$-module. 
		If $\tau=(\mathcal{T},\mathcal{F})$ 
		is a pair of classes of modules in $\sm$, 
		then $\tau$ is a hereditary torsion theory in $\sm$ 
		if and only if 
		there exists a hereditary torsion theory 
		$(\mathfrak{T},\mathfrak{F})$ in $R$-Mod 
		such that $\mathcal{T}=\mathfrak{T}\cap\sm$ and 
		$\mathcal{F}=\mathfrak{F}\cap\sm$.
	\end{lemma}
	
	\begin{proof}
		Suppose that $\tau=(\mathcal{T},\mathcal{F})$ 
		is a hereditary torsion theory in $\sm$. 
		Let $E$ be an $M$-injective module $E$ satisfying 
		the conditions of (\ref{11}), 
		and let $Q$ be the injective hull of $E$ in $R$-Mod. 
		Then $Q$ cogenerates a hereditary torsion theory 
		$(\mathfrak{T},\mathfrak{F})$ in $R$-Mod. 
		We claim that $\mathcal{T}=\mathfrak{T}\cap\sm$ and 
		$\mathcal{F}=\mathfrak{F}\cap\sm$. 
		Since $E\leq Q$, 
		it is clear that $\mathfrak{T}\cap\sm\subseteq \mathcal{T}$. 
		Suppose that there is a homomorphism $f:N\to Q$,
		for $N\in\mathcal{T}$. 
		Since $N$ is in $\sm$, 
		the image $f(M)$ is in $\sm$, 
		which implies that $f(M)\subseteq E$, and so $f=0$. 
		Thus $\mathcal{T}\subseteq\mathfrak{T}\cap\sm$. 
		On the other hand, since $\prod_I^{[M]}E\leq Q^I$ for any set $I$, 
		we have that $\mathcal{F}\subseteq \mathfrak{F}\cap\sm$. 
		Now suppose that $N$ is in $\mathfrak{F}\cap\sm$, 
		so that there is a monomorphism $\alpha:N\to Q^I$. 
		For every index $i\in I$, 
		consider $\pi_i\alpha:N\to Q$, 
		where $\pi_i:Q^I\to Q$ is the canonical projection. 
		Since $N$ is in $\sm$ and the largest submodule of $Q$ in $\sm$ is $E$, 
		we have $\pi_i\alpha(N)\subseteq E$ for all $i\in I$. 
		This implies that $\alpha(N)\subseteq E^I$. 
		Since $\alpha(N)$ is in $\sm$, 
		$\alpha(N)$ must be contained in $\prod_I^{[M]}E$, 
		and thus $N$ is in $\mathcal{F}$. 
		
		Conversely, let $(\mathfrak{T},\mathfrak{F})$ 
		be a hereditary torsion theory in $R$-Mod. 
		Since $\sm$ is a hereditary pretorsion class,
		it follows easily that $(\mathfrak{T}\cap\sm,\mathfrak{F}\cap\sm)$ 
		is a hereditary torsion theory in $\sm$.	
	\end{proof}
	
	Let $\tau$ be a hereditary torsion theory in $\sm$, 
	and let $N$ be a module in $\sm$. 
	Recall that the set of \emph{$\tau$-saturated submodules} of $N$ 
	is defined as 
	\[\Sat_\tau(N)=\{L\leq N\mid N/L\;\text{is $\tau$-torsionfree}\}.\]

	\begin{prop}\label{annihilatorsat}
		Let $\tau$ be a hereditary torsion theory in $\sm$ 
		cogenerated by an $M$-injective module $E$ 
		and let $N\in\sm$. 
		Then 
		\begin{enumerate}
			\item $N$ is $\tau$-torsionfree if and only if 
			$\Ker f\in\Sat_\tau(N)$ for all $f\in\End_R(N)$, and
			\item $\Sat_\tau(N)
			=\{\bigcap_{f\in X}\Ker f\mid X\subseteq\Hom_R(N,E)\}$.
		\end{enumerate}
	\end{prop}
	
	\begin{proof}
		(1) 
		Suppose that $N$ is $\tau$-torsionfree, 
		and let $f:N\to N$ be any endomorphism of $N$. 
		Then there exists an embedding $N/\Ker f\hookrightarrow N$, 
		and since $N$ is $\tau$-torsionfree, we have $\Ker f\in \Sat_\tau(N)$. 
		The converse is obvious.
		
		(2) 
		It is clear that 
		any intersection of kernels of elements of $\Hom_R (N,E)$ 
		belongs to $\Sat_\tau(N)$. 
		On the other hand, if $L\in\Sat_\tau(N)$, 
		then there is a monomorphism 
		$\alpha:N/L\to\prod_I^{[M]}E$ for some index $I$. 
		Let $\pi:N\to N/L$ and 
		$\rho_i:\prod_I^{[M]}E\to E$ be the canonical projections. 
		Then $\rho_i\alpha\pi\in\Hom_R(N,E)$ 
		for all $i\in I$ and $L\subseteq\bigcap_{i\in I}\Ker \rho_i\alpha\pi$. 
		For $x\in\bigcap_{i\in I}\Ker \rho_i\alpha\pi$, 
		we have $\rho_i(\alpha\pi(x))=0$ for all $i\in I$. 
		This implies that $\alpha(\pi(x))=0$, 
		and since $\alpha$ is a monomorphism, 
		we have $\pi(x)=0$, that is, $x\in L$. 
		Thus $L=\bigcap_{i\in I}\Ker \rho_i\alpha\pi$.
	\end{proof}
	
	Recall that given a hereditary torsion theory $\tau$ in $\sm$, 
	a module $N\in\sm$ is said to be 
	\emph{$(M,\tau)$-injective} 
	if $N$ is injective with respect to every exact sequence 
	$0\to K\to V\to V/K\to 0$ in $\sm$ 
	such that $V/K$ is $\tau$-torsion 
	(see \cite[p.\ 61]{wisbauermodules}).
	
	On the other hand, given $N\leq M$, the \emph{$\tau$-purification of $N$ in $M$} 
	is the least submodule $\overline{N}\leq M$ 
	containing $N$ such that $\overline{N}\in\Sat_\tau(M)$. 
	Some properties of the operator $\overline{(\cdot)}$ 
	can be found in \cite[Lemma 2.2]{PepeKrull}.
	
	\begin{rem}\label{essintpur}
	If $M$ is $\tau$-torsionfree,
	then $\overline{N}$ is an essential extension of $N$.
	To see this, 
	suppose that $N^{\prime} \subseteq \overline{N}$ 
	with $N^{\prime} \cap N = 0$.
	Then $N^{\prime}$ is isomorphic to a submodule of $\overline{N} / N$,
	so it is both $\tau$-torsionfree and $\tau$-torsion, 
	and therefore $N^{\prime} = 0$.
	\end{rem}  
	
	\begin{prop}\label{sat}
		Let $\tau$ be a hereditary torsion theory in $\sm$. 
		Suppose that $N\in\sm$ is $\tau$-torsionfree and $(M,\tau)$-injective. 
		Then a submodule $L\leq N$ belongs to $\Sat_\tau(N)$ 
		if and only if $L$ is $(M,\tau)$-injective.
	\end{prop}
	
	\begin{proof}
		Let $L\in\Sat_\tau(N)$ and let $K\leq V\in\sm$ 
		such that $V/K$ is $\tau$-torsion. 
		Consider the following commutative diagram
		\[\xymatrix{0\ar[r] & 
			\Hom_R(V,L)\ar[r] \ar[d]_\lambda & 
			\Hom_R(V,N)\ar[r] \ar[d]^\cong & 
			\Hom_R(V,N/L)\ar[d]^\mu \\ 0\ar[r] & 
			\Hom_R(K,L)\ar[r] & 
			\Hom_R(K,N)\ar[r] & 
			\Hom_R(K,N/L)}\]
		Since $L$ is $\tau$-torsionfree, 
		$\lambda$ is a monomorphism,
		and since $N/L$ is $\tau$-torsionfree, 
		$\mu$ is a monomorphism. 
		It follows from the five lemma that $\lambda$ is an epimorphism, 
		and so $L$ is $(M,\tau)$-injective.
		
		Conversely, 
		let $L\leq N$ be an $(M,\tau)$-injective submodule of $N$, 
		and suppose that $N/L$ is not $\tau$-torsionfree. 
		Then there exists a submodule $\overline{L}\leq N$ 
		such that $\overline{L}/L=\tau\left( N/L\right)$. 
		Note that $\overline{L}\in\Sat_\tau(N)$. 
		By \cite[9.11]{wisbauermodules} there is a split exact sequence 
		\[ 0\to L\to \overline{L}\to \overline{L}/L\to 0 . \] 
		It follows from Remark \ref{essintpur} that $L=\overline{L}$, and thus $L=\overline{L}\in\Sat_\tau(N)$.
		\end{proof}

	Let $\tau$ be a hereditary torsion theory in $\sm$ and let $N\in\sm$. 
	The set of \emph{$\tau$-dense submodules} of $N$ is defined as
	\[\mathcal{L}(N,\tau)=\{K\leq N\mid N/K\;\text{is $\tau$-torsion}\}.\]

	\begin{rem}\label{linj}
		It follows from \cite[9.11(c)]{wisbauermodules} that 
		if $M$ is a generator of $\sm$, 
		then a module $N\in\sm$ is $(M,\tau)$-injective 
		if and only if 
		$N$ is $\mathcal{L}(M,\tau)$-injective.
	\end{rem}
	
	%
	
	\begin{lemma}\label{fgdens}
		Let $\tau$ be a hereditary torsion theory in $\sm$. 
		Let $M$ be $\tau$-torsionfree 
		and let $f,g\in\End_R(M)$ 
		such that $f(M),g(M)\in\mathcal{L}(M,\tau)$. 
		Then $f(g(M))\in\mathcal{L}(M,\tau)$.
	\end{lemma}
	
	\begin{proof}
		We first show that we cannot have $g(M)\subseteq \Ker f$. 
		If this were the case,
		then $\Ker f$ would belong to $\mathcal{L}(M,\tau)$, 
		as well as $\Sat_\tau(M)$ ($M$ is $\tau$-torsonfree). 
		This would give us $\Ker f = M$, forcing $f(M) = (0)$,
		which contradicts the hypothesis that $f(M)\in\mathcal{L}(M,\tau)$. 
		
		Assume that $g(M)\nsubseteq \Ker f$. 
		Note that $f$ induces an epimorphism $M/g(M)\to f(M)/f(g(M))$. 
		Since $M/g(M)$ is $\tau$-torsion, 
		it follows that $f(M)/f(g(M))$ is $\tau$-torsion, 
		and so there is an exact sequence
		\[0\to f(M)/f(g(M))\to M/f(g(M))\to M/f(M)\to 0.\]
		Since the $\tau$-torsion class is closed under extensions, 
		we have that $M/f(g(M))$ is $\tau$-torsion, 
		that is, $f(g(M))\in\mathcal{L}(M,\tau)$.
	\end{proof}
	
	\begin{prop}\label{monoL}
		Let $M$ be projective in $\sm$ and $\tau$-torsionfree. 
		Suppose that $\Sat_\tau(M)$ satisfies the ascending chain condition (ACC). 
		If $f\in\End_R(M)$ and $f(M)\in\mathcal{L}(M,\tau)$,
		then $f$ is a monomorphism.
	\end{prop}
	
	\begin{proof}
		Let $f\in\End_R(M)$ such that $f(M)\in\mathcal{L}(M,\tau)$. 
		Consider the chain
		\[\Ker f\subseteq \Ker(f\circ f)\subseteq\Ker(f\circ f\circ f)\subseteq \cdots\subseteq \Ker(f^{n})\subseteq\cdots\]
		Since $M$ is $\tau$-torsionfree, 
		$\Ker(f^n)\in\Sat_\tau(M)$ for all $n>0$. 
		By hypothesis, there exists $k>0$ such that 
		$\Ker(f^k)=\Ker(f^{k+i})$ for all $i\geq 0$. 
		
		Let $T = \left \{ t \in M \mid f^k (t) \in \Ker f \cap f^k (M) \right \}$.
		Then
		\[0=f(\Ker f\cap f^k(M))=ff^k(T)=f^{k+1}(T),\]
		and it follows that $T\subseteq\Ker f^{k+1}=\Ker f^{k}$,
		so $\Ker f\cap f^{k}(M)=f^{k}(T)=0$. 
		This implies that $\Ker f$ can be embedded into $M/f^k(M)$, 
		which is $\tau$-torsion because $f^{k}(M)\in\mathcal{L}(M,\tau)$ 
		by Lemma \ref{fgdens}. 
		Since $\Ker f \in \Sat_\tau(M)$,
		we have that $\Ker f=0$, 
		and hence $f$ is a monomorphism.
	\end{proof}
	
	\begin{rem}
		Lemma \ref{fgdens} and Proposition \ref{monoL} are generalizations of 
		\cite[Lemma 3.11 and Proposition 3.12]{PepeKrull}, respectively.
	\end{rem}
	
	\begin{rem}
		From the proof of Proposition \ref{monoL}, we can see that if $\tau$ is a hereditary torsion theory, $M$ is $\tau$-torsionfree and $\Sat_\tau(M)$ satisfies ACC, then $M$ is \emph{strongly Hopfian} according to \cite{hmaimou2007generalized}.
	\end{rem}

	\section[The quotient category]{The quotient category of $\sigma [M]$ 
		with $M$ a semiprime Goldie module}\label{qcsgm}
	
	Let $M$ be a module that is projective in $\sm$,
	and consider the hereditary torsion theory 
	$\tau_g=(\mathcal{T},\mathcal{F})$ in $\sm$ 
	cogenerated by $\widehat{M}=E^{[M]}(M)$. 
	Then
	\begin{equation}\label{12}
		\begin{split}
			\mathcal{T}& =\{N\in\sm\mid \Hom_R(N,\widehat{M})=0\} \\
			\mathcal{F}& =\left\lbrace L\in\sm\mid L\hookrightarrow 
			\textstyle{\prod_I^{[M]}\widehat{M}}\;
			\text{for some index set}\;I\right\rbrace 
		\end{split}
	\end{equation}
	
	Recall that a module $N\in\sm$ is said to be \emph{$M$-singular} 
	if there exists an exact sequence in $\sm$
	\[0\to K\to L\to N\to 0\]
	such that $K\ess L$. 
	The class of all $M$-singular modules in $\sm$ 
	is a hereditary pretorsion class that will be denoted by $\mathcal{S}$. 
	Note that every module $N$ in $\sm$ 
	contains a largest $M$-singular submodule $\mathcal{S}(N)$. 
	The module $N$ is said to be \emph{non-$M$-singular} if $\mathcal{S}(N)=0$. 
	It follows from \cite[10.2]{wisbauermodules} that if $M$ is non-$M$-singular, $\mathcal{T}=\mathcal{S}$ and $\mathcal{F}$ 
	consists of all non-$M$-singular modules in $\sm$ .
	
	\begin{lemma}\label{injtg}
		Let $M$ be projective in $\sm$ 
		and let $\mathcal{S}$ be the 
		hereditary pretorsion class of $M$-singular modules. 
		Then $N\in\sm$ is $M$-injective 
		if and only if 
		it is $(M,\mathcal{S})$-injective.
	\end{lemma}
	
	\begin{proof}
		$\Rightarrow$)
		This is clear.
		
		$\Leftarrow$)
		Let $L$ be a submodule of $M$ and let $f:L\to N$ be any homomorphism. 
		Without loss of generality, we can assume that $L\ess M$, 
		and hence $M/L\in \mathcal{S}$. 
		Since $N$ is $(M,\mathcal{S})$-injective, 
		$f$ can be extended to a homomorphism $M\to N$.
	\end{proof}
	
	\begin{defn}
		Let $M$ be a module. 
		An \emph{$M$-annihilator} is a submodule of $M$ 
		of the form $\bigcap_{f\in X}\Ker f$, for some $X\subseteq\End_R(M)$.
	\end{defn}
	
	\begin{defn}
		A module $M$ is called a \emph{Goldie module} 
		if $M$ satisfies the ascending chain condition (ACC) on $M$-annihilators 
		and has finite uniform dimension.
	\end{defn}
	
	Given two submodules $N$ and $L$ of a module $M$, 
	their \emph{product} in $M$ is the submodule given by 
	\[N_ML=\sum\{f(N)\mid f\in\Hom_R(M,L)\}.\]
	This product was defined in \cite{BicanPr},
	and can be used to define generalizations
	of the notions of prime and semiprime ideals of a ring,
	as follows. 
	A fully invariant submodule $K\leq M$ 
	is said to be a \emph{prime} submodule of $M$
	if $N_ML\subseteq K$ implies that $N\subseteq K$ or $L\subseteq K$,
	for all fully invariant submodules $N, L$ of $M$;
	it is said to be a \emph{semiprime} submodule of $M$
	if $N_MN\subseteq K$ implies that $N\subseteq K$,
	for all fully invariant submodules $N\subseteq M$.
	If $0$ is a (semi)prime submodule of a module $M$, 
	we say that $M$ is a \emph{(semi)prime module}. 
	These concepts were introduced in \cite{raggiprime} and \cite{raggisemiprime}, 
	and semiprime Goldie modules were studied later in \cite{maugoldie}. 
	
	An operator $\mathfrak{L}$ was defined in \cite{mauprimerad}, 
	which assigns to each module $M$ projective in $\sm$ 
	its \emph{prime radical} $\mathfrak{L}(M)$ \cite[Corollary 3.12]{mauprimerad}, that is, the intersection of all prime submodules of $M$. 
	This operator commutes with direct sums \cite[Lemma 3.14]{mauprimerad} 
	and satisfies the condition that $\mathfrak{L}(M/\mathfrak{L}(M))=0$ 
	\cite[Proposition 3.9]{mauprimerad}. 
	
	\begin{defn}
		A module $M$ is said to be \emph{retractable} 
		if $\Hom_R(M,N)\neq 0$ for every $0\neq N\leq M$.
	\end{defn}
	
	\begin{rem}
		If $M$ is a semiprime module that is projective in $\sm$, 
		then $M$ is retractable. 
		To show this, suppose that $\Hom_R(M,N)=0$ for some $N\leq M$. 
		Then $N_MN=0$, and so $N=0$. 
		Since $M$ is semiprime and projective in $\sm$.
	\end{rem}
	
	The following result was given as a generalization of Goldie's Theorem in \cite[Theorem 2.8]{maugoldie}.
	
	\begin{thm}
		Let $M$ be projective in $\sm$ with finite uniform dimension. The following conditions are equivalent:
		\begin{enumerate}[label=\emph{(\alph*)}]
			\item $M$ is semiprime and non-$M$-singular.
			\item $M$ is semiprime and satisfies ACC on annihilators.
			\item Let $N\leq M$, then $N$ is essential in $M$ if and only if there exists a monomorphism $f:M\to N$.
		\end{enumerate}
	\end{thm}	
	
	The proof of (b)$\Rightarrow$(c) was given as a consequence of \cite[Proposition 3.13]{PepeKrull} but it turns out that the proof of that result is not correct. For the convenience of the reader we prove (b)$\Rightarrow$(c) below.
	
	\begin{lemma}\label{spnmskern0}
		Let $M$ be projective in $\sm$. Suppose that $M$ is a semiprime non $M$-singular module with finite uniform dimension and $0\neq N\leq M$. Then there exists $0\neq f:M\to N$ such that $\Ker f\cap N=0$.
	\end{lemma}
		
	\begin{proof}
		Let $N\leq M$. We will proceed by induction on $Udim(N)=\ell$. Since $M$ is semiprime, $N_MN\neq 0$. Hence there exist $f:M\to N$ and $n\in N$ such that $f(n)\neq 0$. If $N\cap \Ker f\neq 0$, then $N\cap \Ker f\ess N$ and $N\cap \Ker f$ is proper in $N$ because $N\nsubseteq \Ker f$. It follows that $\frac{N}{N\cap \Ker f}$ is $M$-singular and there exists a nonzero monomorphism $\frac{N}{N\cap \Ker f}\to N$ which is a contradiction since $N$ is non $M$-singular. Thus $N\cap \Ker f=0$. Now suppose that $Udim(N)=\ell$ and fix $L\leq N$ with $Udim(L)=\ell-1$. By induction hypothesis, there exists $0\neq g:M\to L$ such that $\Ker g\cap L=0$. We can assume that $U=\Ker g\cap N\neq 0$, otherwise we can take the morphism $g$. Then $Udim(U)=1$ and as we did before there exists $0\neq h:M\to U$ such that $\Ker h\cap U=0$. Note that if, $(g+h)(x)=0$, then $g(x)=-h(x)\in U\cap L=0$. Let $f=g+h:M\to N$. Thus $\Ker f\cap N=\Ker(g+h)\cap N=\Ker g\cap \Ker h\cap N=\Ker h\cap U=0$.
	\end{proof}
	
	\begin{prop}
		Let $M$ be projective in $\sm$. If $M$ is a semiprime Goldie module then, a submodule $N$ is essential in $M$ if and only if there exists a monomorphism $f:M\to N$.
	\end{prop} 
	
	\begin{proof}
		Let $N\leq{M}$. By \cite[Proposition 3.4]{PepeKrull}, $M$ is non $M$-singular. Suppose that $N\leq_e{M}$. It follows from Lemma \ref{spnmskern0} that there exists $0\neq f:M\to N$ such that $\Ker f\cap N=0$. Since $N\ess M$, $f$ must be a monomorphism. Now, if $f:M\rightarrow{N}$ is a monomorphism then $N\leq_e{M}$ because $M$ has finite uniform dimension.	
	\end{proof}
	
	Let $M$ be a semiprime Goldie module that is projective in $\sm$. Then $M$ is non-$M$-singular by \cite[Theorem 2.8]{maugoldie}. 
	Consider the quotient category $\sm/\tau_g$ and the localization functor
	\[\mathcal{Q}_{\tau_g}:\sm\to \sm/\tau_g.\]
	In this case, $\mathcal{Q}_{\tau_g}$ is an exact functor 
	given by $\mathcal{Q}_{\tau_g}(N)=E^{[M]}(N/\gamma(N))$,
	where $\gamma$ is the torsion radical 
	associated to $\tau_g$ \cite[9.13, 10.5]{wisbauermodules}. 
	Since every object in $\sm/\tau_g$ is $M$-injective, 
	the category $\sm/\tau_g$ is a spectral category. 
	
	As a particular case of Proposition~\ref{sat}, 
	we have the following lemma.
	
	\begin{lemma}\label{sattg}
		Let $M$ be projective in $\sm$ and non-$M$-singular, 
		and suppose that $N\in\sm$ is $\tau_g$-torsionfree and $(M,\tau_g)$-injective. 
		Then $L\in\Sat_{\tau_g}(N)$ if and only if 
		$L$ is $\tau_g$-torsionfree and $(M,\tau_g)$-injective.
	\end{lemma}

	\begin{lemma}\label{sataccudim}
		Let $\tau$ be a hereditary torsion theory in $\sm$, 
		and let $N$ be a $\tau$-torsionfree module in $\sm$. 
		If $\Sat_\tau(N)$ satisfies the ascending chain condition, 
		then $N$ has finite uniform dimension.
	\end{lemma}
	
	\begin{proof}
		Suppose that $N$ contains an infinite direct sum of nonzero submodules.
		As in \cite[Proposition~IX.1.11]{stenstromrings},
		the quotient functor $\mathcal{Q}_{\tau}$ 
		is a left adjoint of the inclusion functor
		from $\sm/\tau$ into $\sm$,
		and so it preserves coproducts.
		But by hypothesis, 
		the object $\mathcal{Q}_{\tau} (N)$ satisfies the ACC 
		in the quotient category,
		and thus it cannot contain an infinite coproduct of subobjects.
		This provides the desired contradiction.
	\end{proof}

	\begin{lemma}\label{specnonms}
		Let $M$ be an $R$-module. The following conditions are equivalent for the torsion theory $\tau_g$.
		\begin{enumerate}[label=\emph{(\alph*)}]
			\item $M$ is non-$M$-singular.
			\item $\sm/\tau_g$ is a spectral category
		\end{enumerate}
	\end{lemma}
	
	\begin{proof}
		(a)$\Rightarrow$(b) $\sm/\tau_g$ consists of all non-$M$-singular injective modules in $\sm$. Thus it is spectral.
		
		(b)$\Rightarrow$(a) Let $N\leq L\leq M$ with $N$ essential in $L$. We have the following diagram with exact rows
		\[\xymatrix{0\ar[r] & N\ar[r]^i\ar[d]_{\varphi_N} & L\ar[d]_{\varphi_L} \\ 
			0\ar[r] & \mathcal{Q}_{\tau_g}(N)\ar[r]^{\mathcal{Q}(i)} & \mathcal{Q}_{\tau_g}(L) }\]
		The vertical arrows are the localization morphims. Since $N$ is essential in $\mathcal{Q}(N)$ and $\varphi_Li$ is an essential monomorphism, then $\mathcal{Q}(i)$ is an essential monomorphim. By hypothesis $\mathcal{Q}(i)$ splits which implies that $\mathcal{Q}(i)$ is an isomorphism. By \cite[9.18]{wisbauermodules} $L/N$ is $\tau_g$-torsion. Let $f:N\to M$ be any morphism with $N\leq M$ and $\Ker f$ essential in $N$. By the above, $N/\Ker f$ is $\tau_g$-torsion and it embeds in $M$. Thus $f=0$ and $M$ is non-$M$-singular.
	\end{proof}

	\begin{lemma}\label{retfactor}
		Let $M$ be projective in $\sm$ and $H$ be a module in $\sm$. If $\Hom_R(M,H)\neq 0$, then $M/\Ann_M(H)$ is retractable.
	\end{lemma}

	\begin{proof}
		Let $K/\Ann_M(H)$ be a submodule of $M/\Ann_M(H)$ and suppose that 
		\[\Hom_R(M/\Ann_M(H),K/\Ann_M(H))=0.\]
		Since $\Ann_M(H)$ is fully invariant in $M$, any homomorphism $f:M\to K$ induces a homomorphism $\overline{f}:M/\Ann_M(H)\to K/\Ann_M(H)$ such that $\pi f=\overline{f}\pi$ where $\pi:M\to M/\Ann_M(H)$ is the canonical projection. This implies that for all $f\in\Hom_R(M,K)$, $f(M)\subseteq \Ann_M(H)$. Therefore $0=tr^M(K)_MH=(M_MK)_MH=M_M(K_MH)=tr^M(K_MH)$. This implies that $K_MH=0$ because $\Hom_R(M,H)\neq 0$. Thus, $K\subseteq \Ann_M(H)$. Hence $M/\Ann_M(H)$ is retractable.
	\end{proof}

	\begin{thm}\label{goldiefl}
		Let $M$ be projective in $\sm$, $S=\End_R(M)$ and $T=\End_R(\widehat{M})$. The following conditions are equivalent:
		\begin{enumerate}[label=\emph{(\alph*)}]
			\item $M$ is a semiprime Goldie module;
			
			\item $S$ is a right order in $T$, $M$ is retractable and $\sm/\tau_g$ is a discrete spectral category equivalent to Mod-$T$ (the category of right $T$-modules) with $\mathcal{Q}_{\tau_g}(M)$ as generator of finite length.
			
			\item $M$ is retractable, $\mathcal{Q}_{\tau_g}(M)$ is a semisimple generator of finite length of  $\sm/\tau_g$ which is  
			equivalent to Mod-$T$ (the category of right $T$-modules). For every simple object $A$ of $\sm/\tau_g$, $\Ann_M(A)$ is a prime submodule of $M$. 
			
			\item $M$ is retractable, $\mathfrak{L}(M)=0$ and $\mathcal{Q}_{\tau_g}(M)$ is semisimple of finite length 
			in the quotient category $\sm/\tau_g$.
			
			\item $M$ is retractable, $\mathcal{Q}_{\tau_g}(M)$ is a semisimple generator of finite length 
			in the quotient category $\sm/\tau_g$, and for every $N\in \mathcal{L}(M,\tau)$ there exists an $f:M\to N$ such that $f(M)\in\mathcal{L}(N,\tau)$.
		\end{enumerate}
	\end{thm}

	\begin{proof}
		(a)$\Rightarrow$(b) 
		Let $\mathcal{Q}_{\tau_g}$ be the localization functor. 
		Then there is an essential monomorphism 
		$\alpha:M^{(X)}\hookrightarrow \widehat{M}^{(X)}
		=\mathcal{Q}_{\tau_g}(M)^{(X)}$, 
		and so by applying $\mathcal{Q}_{\tau_g}$, we get that 
		$\mathcal{Q}_{\tau_g}(M^{(X)})\cong\mathcal{Q}_{\tau_g}(M)^{(X)}$.  
		Since every $M$-injective module in $\sm$ is $M$-generated 
		and the localization functor is exact, 
		every object in the category $\sm/\tau_g$ is generated by $\widehat{M}$. 
		Since $\sm/\tau_g$ is a spectral category with generator $\widehat{M}$, 
		it follows from \cite[Ch. XII, Theorem 1.3 and Theorem 1.4]{stenstromrings} 
		that $\sm/\tau_g$ is a discrete spectral category and it is equivalent to Mod-$T$ via the functor $\Hom(\widehat{M},\_)$. 
		Since $M$ is a semiprime Goldie module, 
		it follows that $S$ is a semiprime right Goldie ring and 
		$T$ is the classical right ring of quotients 
		of $S$ \cite[Theorem 2.22]{maugoldie}. 
		
		(b)$\Rightarrow$(c) Since $\sm/\tau_g$ is a discrete spectral category, $\mathcal{Q}_{\tau_g}(M)$ is a semisimple object. By hypothesis, $\mathcal{Q}_{\tau_g}(M)$ is a generator of finite length. By Lemma \ref{specnonms}, $M$ is non-$M$-singular and hence $\mathcal{Q}_{\tau_g}=\widehat{M}$ and $T$ is a regular ring.
		
		Claim 1: If $N$ is a fully invariant submódule of $M$, then $E^{[M]}(N)$ is fully invariant in $\widehat{M}$. By hypothesis $S$ is a right order in $T$, that is, there exists a subset $C$ of regular elements of $S$ such that $T=SC^{-1}$. Let $f\in C$. Then $f(N)\subseteq N$. It follows that $E^{[M]}(f(N))\subseteq E^{[M]}(N)$. Considering $f^{-1}\in T$, we have 
		\[N\subseteq f^{-1}\left( E^{[M]}(f(N))\right)\subseteq f^{-1}\left( E^{[M]}(N)\right).\]
		This implies that $E^{[M]}(N)\subseteq f^{-1}\left( E^{[M]}(N)\right)$ because $T$ is regular and hence there is an ascending chain of direct summands of $\widehat{M}$:
		\[E^{[M]}(N)\subseteq f^{-1}\left( E^{[M]}(N)\right)\subseteq f^{-2}\left( E^{[M]}(N)\right)\subseteq\cdots\]
		Since direct summands are $\tau_g$-torsionfree, they are subobjects of $\mathcal{Q}_{\tau_g}(M)$ and so this chain must stabilize. Applying $f$ many times as we need,  $f^{-1}\left(E^{[M]}(N)\right)=E^{[M]}(N)$. Now, given any $g\in S$ we get the following commutative squares
		\[\xymatrix{N\ar@{^(-{>}}[r]^i \ar[d]_{g|} & M \ar[d]^g & & & E^{[M]}(N)\ar@{^(-{>}}[r]^-{\mathcal{Q}(i)} \ar[d]_{\hat{g}|}^{\mathcal{Q}(g|)} & \widehat{M} \ar[d]^{\hat{g}} \\
			N\ar@{^(-{>}}[r]^i & M & & & E^{[M]}(N)\ar@{^(-{>}}[r]_-{\mathcal{Q}(i)} & \widehat{M}}\]
		
		where $i$ is the canonical inclusion and $\hat{g}$ is the image of $g$ in $T$. We have that $\mathcal{Q}(i)$ is the inclusion and therefore $\hat{g}\left(E^{[M]}(N)\right)\subseteq E^{[M]}(N)$. Since every morphism in $T$ has the form $\hat{g}f^{-1}$ with $f\in C$ and $g\in S$, then $E^{[M]}(N)$ is fully invariant in $\widehat{M}$.
		
		Let $\mathcal{Q}_{\tau_g}(M)=H_1\oplus\cdots\oplus H_\ell$ be a decomposition in homogeneous components of $\mathcal{Q}_{\tau_g}(M)$ in $\sm/\tau_g$. Then $N=\left(\bigoplus_{i\neq j}H_j\right)\cap M$ is a fully invariant $\tau_g$-pure submodule of $M$. Let $L/N$ be a fully invariant submodule of $M/N$. Then $L$ is fully invariant in $M$. By Claim 1, $\mathcal{Q}_{\tau_g}(L)=E^{[M]}(L)$ is fully invariant in $\widehat{M}$ and therefore $\mathcal{Q}_{\tau_g}(L/N)=\mathcal{Q}_{\tau_g}(L)/\mathcal{Q}_{\tau_g}(N)$ is fully invariant in $\mathcal{Q}_{\tau_g}(M/N)\cong H_i$. Since $H_i$ is a homogeneous component, it has no nontrivial fully invariant submodules. Hence $\mathcal{Q}_{\tau_g}(L/N)=H_i$ and this implies that every fully invariant submodule of $M/N$ is $\tau_g$-dense.
		
		Claim 2: $\Ann_M(H_i)=\left(\bigoplus_{i\neq j}H_j\right)\cap M$. Let $N=\left(\bigoplus_{i\neq j}H_j\right)\cap M$. Every morphism $f:M\to H_i$ corresponds uniquely to a morphism $\hat{f}:\widehat{M}\to H_i$. So, $N\subseteq \Ker f$. Thus $N\subseteq\Ann_M(H_i)$. On the other hand, $\Ann_M(H_i)/N$ is a fully invariant submodule of $M/N$. Then it is $\tau_g$-dense. But $\Ann_M(H_i)\in\Sat_{\tau_g}(M)$. Hence $\Ann_M(H_i)=N$.
		
		Claim 3: $N=\left(\bigoplus_{i\neq j}H_j\right)\cap M$ is a prime submodule of $M$. Consider two fully invariant submodules $A$ and $B$ of $M/N$ such that $A_{M/N} B=0$. Note that $M/N$ is $\tau_g$-torsionfree and so is $B$. Then, for every morphism $f:M/N\to B$, $A\subseteq \Ker f$. Since $A$ is $\tau_g$-dense, so is $\Ker f$. Thus $f=0$. Note that $\Hom_R(M,H_i)\neq 0$ because $H_i$ is an injective hull of a submodule of $M$. Then $M/N$ is retractable by Claim 2 and Lemma \ref{retfactor}. Thus $B=0$ and hence $M/N$ is a prime module which is equivalent to say that $N$ is a prime submodule of $M$.
		
		Let $A$ be a simple object in $\sm/\tau_g$. Then there exists a homogeneous component $H$ of $\mathcal{Q}_{\tau_g}(M)$ and a monomorphism $A\to H$. Then $\Ann_M(H)\subseteq \Ann_M(A)$. It follows that $\Ann_M(A)/\Ann_M(H)$ is fully invariant in $M/\Ann_M(H)$ and hence it is $\tau_g$-dense. Thus $\Ann_M(H)=\Ann_M(A)$ which is prime in $M$.
		
		(c)$\Rightarrow$(d) By hypothesis, $\mathcal{Q}_{\tau_g}(M)$ is semisimple of finite length in $\sm/\tau_g$. Let $H_1,.., H_\ell$ be the homogeneous components of $\mathcal{Q}_{\tau_g}(M)=\widehat{M}$. Then $\Ann_M(H_i)=P_i$ is a prime submodule of $M$. Therefore 
		
		\[\mathfrak{L}(M)\subseteq\bigcap_{i=1}^\ell\Ann_M(H_i)=\Ann_M\left(\sum_{i=1}^\ell H_i\right)=\Ann_M(\widehat{M})=0.\]
		
		(c)$\Rightarrow$(a) 
		Since $\mathfrak{L}(M)=0$, $M$ is a semiprime module.
		Note that for every $X\subseteq\End_R(M)$, the annihilator
		$\bigcap_{f\in X}\Ker f$ belongs to $\Sat_{\tau_g}(M)$.
		By Proposition \ref{sat}, 
		$\Sat_{\tau_g}(M)$ corresponds to the subobjects of 
		$\mathcal{Q}_{\tau_g}(M)$. 
		Since $\mathcal{Q}_{\tau_g}(M)$ has finite length, 
		$M$ must satisfy the ACC on $M$-annihilators, 
		and it has finite uniform dimension by Lemma~\ref{sataccudim}.
		
		(a)$\Rightarrow$(d)
		Since $M$ is a semiprime Goldie module, 
		it is non-$M$-singular and retractable. 
		Furthermore, we have that $\mathcal{Q}_{\tau_g}(M)=\widehat{M}$. 
		It follows from Lemma \ref{sattg} 
		that the subobjects of $\mathcal{Q}_{\tau_g}(M)$ 
		corresponds to the direct summands of $\widehat{M}$. 
		By \cite[Theorem 2.22]{maugoldie}, 
		$\End_R(M)$ is an order in the semisimple Artinian ring 
		$T=\End_R(\widehat{M})$, 
		and then the subobjects of $\mathcal{Q}_{\tau_g}(M)$ 
		are in one-to-one correspondence with the idempotents in $T$. 
		Since $T$ is Artinian, 
		it cannot contains an infinite set of distinct idempotents, 
		and therefore $\mathcal{Q}_{\tau_g}(M)$ has finite length. 
		
		If $M$ is a semiprime Goldie module, $\mathcal{L}(M,\tau)$ is exactly the set of essential submodules of $M$. Given $N\in\mathcal{L}(M,\tau)$, there exists a monomorphism $f:M\to N$ by \cite[Theorem 2.8]{maugoldie}. Since $M$ has finite uniform dimension $f(M)$ es essential in $N$. Thus $f(M)\in\mathcal{L}(M,\tau)$. 
		
		(d)$\Rightarrow$(a) By Lemma \ref{specnonms}, $M$ is non-$M$-singular. This implies that every essential submodule of $M$ is in $\mathcal{L}(M,\tau_g)$. With Lemma \ref{monoL} and \ref{sataccudim} we can see that $M$ satisfies the condition (3) of \cite[Theorem 2.8]{maugoldie}. Thus $M$ is a semiprime Goldie module.
	\end{proof}
	
	\begin{cor}\label{qmax}
		The following conditions are equivalent for a ring $R$:
		\begin{enumerate}[label=\emph{(\alph*)}]
			\item $R$ is a semiprime left Goldie ring.
			
			\item $Q_{max}^\ell(R)$ is semisimple and $R$ is a left order in $Q_{max}(R)$. 
			
			\item $Q_{max}^\ell(R)$ is semisimple and for every simple $Q_{max}(R)$-module $A$, $\Ann_R(A)$ is a prime ideal of $R$.
			
			\item $Q_{max}^\ell(R)$ is semisimple and $\mathfrak{L}(R)=0$.
			
			\item $Q_{max}^\ell(R)$ is semisimple and for every dense left ideal $I$, there exists $a\in I$ such that $Ra$ is dense.
		\end{enumerate} 
	\end{cor}
	
	\begin{rem}
		Last Corollary extends \cite[13.40 and 13.41]{lamlectures} and therefore, we have that $Q_{max}^\ell(R)=Q_{cl}^\ell(R)$ in items (b) to (e).
	\end{rem}
	
	\section{Modules with finite reduced rank in $\sm$}\label{mfrr}
	
	We are now ready to study the notion of reduced rank in $\sigma [M]$.
	To obtain results parallel to those for $R$-Mod,
	we need to assume that $M$ is a projective generator in $\sigma [M]$,
	thus not deviating too far from a category of modules.
	Before formally stating the definition of reduced rank in $\sigma [M]$,
	we need to recall some results that are used subsequently.
	
	\begin{lemma}[{\cite[Proposition 1.5]{PepeFbn}}]\label{genq}
		Let $M$ be a projective generator of $\sm$. 
		The following conditions hold for a fully invariant submodule $P\leq M$:
		\begin{enumerate}
			\item $\sigma[M/P]=\{N\in\sm\mid P\subseteq \Ann_M(N)\}$;
			\item $M/P$ is a projective generator of $\sigma[M/P]$.
		\end{enumerate}
	\end{lemma}

	%
	%
	
	\begin{lemma}[C. N\u{a}st\u{a}sescu, {\cite[Corollarie 2]{nuastuasescu1980theoreme}}]\label{Nas}
		If $\mathcal{G}$ is a Grothendieck category having an Artinian generator, 
		then any Artinian object in $\mathcal{G}$ is Noetherian.
	\end{lemma}
	
	\begin{defn}
		Let $M$ be a module with prime radical $\mathfrak{L}(M)$. 
		We say that $M$ has \emph{finite reduced rank in $\sm$} 
		if the module of quotients $\mathcal{Q}_\gamma(M)$ 
		has finite length in the quotient category $\sm/\gamma$ cogenerated by $E^{[M]}(M/\mathfrak{L}(M))$.
	\end{defn}
	
	\begin{rem}
		For a ring $R$, 
		having finite reduced rank in $R$-Mod 
		coincides with the definition given in \cite{beachy1982rings}. 
		When we say a ring $R$ has finite reduced rank, 
		it will mean that $_RR$ has finite reduced rank in $\sigma[R]=R$-Mod.
		Note that Theorem~\ref{rrt} is a direct generalization of
		\cite[Theorem~1]{beachy1982rings}.
	\end{rem}
	
	For the prime radical $\mathfrak{L}(M)$ of a module $M$,
	we will use the notation $\mathfrak{L}(M)^k$ 
	to denote the $k^{th}$ power of $\mathfrak{L}(M)$ 
	with respect to the product $-_M-$.
	
	\begin{thm}\label{rrt}
		Let $M$ be a module projective in $\sm$ with prime radical $\mathfrak{L}(M)$,
		and let $\gamma$ be the hereditary torsion theory in $\sm$ 
		cogenerated by $E^{[M]}(M/\mathfrak{L}(M))$. 
		Consider the following conditions:
		\begin{enumerate}
			\item $M$ has finite reduced rank in $\sm$;
			\item the set $\Sat_\gamma(M)$ satisfies the ascending chain condition;
			\item the following conditions hold:
			\begin{enumerate}[label=(\roman*)]
				\item $M/\mathfrak{L}(M)$ is a semiprime Goldie module;
				\item $\mathfrak{L}(M)^k\subseteq \gamma(M)$ for some $k>0$;
				\item for any $A\in\Sat_\gamma(M)$, 
				the module $M/A$ has finite uniform dimension.
			\end{enumerate}
		\end{enumerate}
		Then \textit{(1)}$\Rightarrow$\textit{(2)}$\Rightarrow$\textit{(3)}, 
		and in addition, if $M$ is a generator of $\sm$, 
		then the three conditions are equivalent.
	\end{thm}
	
	\begin{proof}
		Writing $E=E^{[M]}(M/\mathfrak{L}(M))$, 
		we have $\gamma=(\mathcal{T},\mathcal{F})$, with 
		\[\mathcal{T}=\{K\in\sm\mid \Hom_R(K,E)=0\}\text{ and}\]
		\[\mathcal{F}=\{L\in\sm\mid L\hookrightarrow \textstyle{\prod_I^{[M]}E}\;
		\text{for some index set}\;I\},\]
		where $\prod_I^{[M]}E$ is the product of $I$ copies of $E$ 
		in the category $\sm$. 
		
		$(1)\Rightarrow(2)$ 
		By Proposition \ref{sat},
		the set $\Sat_\gamma(M)$ corresponds to 
		the subobjects of $\mathcal{Q}_\gamma(M)$ 
		in the quotient category $\sm/\gamma$. 
		Condition (2) holds since by hypothesis
		$\mathcal{Q}_\gamma(M)$ has finite length in $\sm/\gamma$.  
		
		$(2)\Rightarrow(3)$ 
		If $A\in\Sat_\gamma(M)$, 
		then the module $M/A$ is $\gamma$-torsionfree 
		and there is an embedding  $\Sat_\gamma(M/A)\hookrightarrow\Sat_\gamma(M)$. 
		By Proposition \ref{annihilatorsat}, 
		each $M/A$-annihilator of $M/A$ is in $\Sat_\gamma(M/A)$, 
		so $M/A$ satisfies the ACC on $M/A$-annihilators. 
		It follows from Lemma~\ref{sataccudim} 
		that $M/A$ has finite uniform dimension, 
		and thus $M/A$ is a Goldie module for every $A\in\Sat_\gamma(M)$. 
		In particular, $M/\mathfrak{L}(M)$ is a semiprime Goldie module. 
		Note that the prime radical of $M/\gamma(M)$ is $\mathfrak{L}(M)/\gamma(M)$. 
		It follows from Corollary \cite[Corollary 5.4]{mauprimerad} 
		that $\mathfrak{L}(M)/\gamma(M)$ is nilpotent, i.e., 
		there exists $k>0$ such that $\mathfrak{L}(M)^{k}\subseteq\gamma(M)$.
		
		Now suppose that $M$ is a generator of the category $\sm$. 
		To show that (3)$\Rightarrow$(1), 
		let $\{A_i\}_{i=1}^\infty$ be a descending chain in $\Sat_\gamma(M)$ 
		with $A=\bigcap_{i=1}^\infty A_i$. 
		By hypothesis, $M/A$ has finite uniform dimension. 
		Let $U\leq M/A$ be any uniform submodule of $M/A$. 
		Since $M/A$ is $\gamma$-torsionfree and $\mathfrak{L}(M)^k\subseteq\gamma(M)$, 
		we have ${\mathfrak{L}(M)^k}_MU=0$. 
		Without loss of generality, 
		we can assume that ${\mathfrak{L}(M)^k}_MU=0$ 
		but ${\mathfrak{L}(M)^{k-1}}_MU\neq 0$, 
		where ${\mathfrak{L}(M)^0}_MU=U$. 
		Let $V$ denote the product ${\mathfrak{L}(M)^{k-1}}_MU\neq 0$, 
		and note that $\mathfrak{L}(M)\subseteq \Ann_M(V)$. 
		Since $M$ is a generator of $\sm$, 
		by Lemma \ref{genq} we have $V\in\sigma[M/\mathfrak{L}(M)]$.
		
		Let $\overline{\gamma}$ be the hereditary torsion theory 
		in $\sigma[M/\mathfrak{L}(M)]$ induced by $\gamma$, that is, 
		$\overline{\gamma}=\left( \mathcal{T}\cap\sigma[M/\mathfrak{L}(M)],
		\mathcal{F}\cap\sigma[M/\mathfrak{L}(M)]\right) $. 
		Then $\overline{\gamma}$ is cogenerated by $\widehat{M/\mathfrak{L}(M)}$,
		and hence $\overline{\gamma}=\chi(M/\mathfrak{L}(M))$ 
		in $\sigma[M/\mathfrak{L}(M)]$. 
		Since $M/\mathfrak{L}(M)$ is a semiprime Goldie module, 
		$\chi(M/\mathfrak{L}(M))$ is the hereditary torsion theory 
		generated by all $M/\mathfrak{L}(M)$-singular modules 
		in $\sigma[M/\mathfrak{L}(M)]$. 
		It follows that $E^{[M/\mathfrak{L}(M)]}(V)$ 
		(the injective hull of $V$ in $\sigma[M/\mathfrak{L}(M)]$) 
		is a $\overline{\gamma}$-cocritical module. 
		This implies that $E^{[M]}(U)$ 
		contains a $\gamma$-cocritical submodule 
		because $E^{[M/\mathfrak{L}(M)]}(V)\subseteq E^{[M]}(V) \subseteq E^{[M]}(U)$. 
		
		By assumption, $E^{[M]}(M/A)\cong\bigoplus_{i=1}^n E^{[M]}(U_i)$,
		with $U_i$ uniform. 
		Since each $E^{[M]}(U_i)$ contains a $\gamma$-cocritical submodule, 
		the socle of $E^{[M]}(M/A)$ in the quotient category $\sm/\gamma$ 
		is essential and of finite length. 
		Thus $E^{[M]}(M/A)$ satisfies the finite intersection property 
		for subobjects in $\sm/\gamma$, 
		and therefore that $A=\bigcap_{i=1}^t A_i$ for some $t>0$,
		since each $A_i$ is in $\Sat_\gamma(M/A)$. 
		It follows that $\Sat_\gamma(M)$ 
		satisfies the descending chain condition (DCC), 
		and hence $\mathcal{Q}_\gamma(M)$ satisfies the DCC on subobjects. 
		Since $M$ is a generator of $\sm$, the quotient module
		$\mathcal{Q}_\gamma(M)$ is a generator of $\sm/\gamma$, 
		so it follows from Lemma \ref{Nas} 
		that $\mathcal{Q}_\gamma(M)$ also satisfies the ACC on subobjects.
	\end{proof}
	
	
	
	It was shown by Lenagan in \cite{lenagan}
	that a ring with Krull dimension on the left has finite reduced rank.
	In order to obtain conditions that imply that an $R$-module $M$ 
	with Krull dimension has finite reduced rank in $\sm$, 
	we need to make some remarks regarding 
	\cite[Lemma 2.14 and Theorem 2.16]{Pepenil}. 
	These results are stated for a finitely generated module $M$, 
	but that condition can be omitted. 
	To show this, we will prove \cite[Lemma 2.14]{Pepenil} 
	without the assumption that $M$ is finitely generated.
	
	\begin{lemma}\label{ObsPPnil}
		Let $M$ be projective in $\sm$. 
		If $I$ and $N$ are fully invariant submodules of $M$ 
		such that $0\neq I\subseteq N$ and $I_MI=0$, 
		then there exists a nonzero fully invariant submodule $A$ of $M$ 
		such that $A$ is maximal with respect $A_MA=0$ and $A\subseteq N$.
	\end{lemma}
	
	\begin{proof}
		Consider the set 
		$\Gamma=\{A\leq M\mid A\;\text{ is fully invariant and }\;A_MA=0\}$, 
		and note that $I\in\Gamma$ by hypothesis. 
		Let $\{A_i\}_I$ be a chain in $\Gamma$ and let $C=\bigcup_IA_i$. 
		Since $M$ is projective in $\sm$, we have
		\[\textstyle C_MC=C_M\left(\bigcup_IA_i\right)=C_M\left(\sum_IA_i\right)
		=\sum_I\left( C_MA_i\right).\]
		Let $c\in C$ and let $f:M\to A_i$ be any morphism. 
		Since $C$ is a chain, 
		there exists $j\in I$ such that $c\in A_j$. 
		Furthermore, $A_j\subseteq A_i$ or $A_i\subseteq A_j$, 
		so suppose that $A_i\subseteq A_j$. 
		Then we can have $f:M\to A_j$, 
		and therefore $f(c)\in {A_j}_MA_j=0$. 
		Now, if $A_j\subseteq A_i$, then $c\in A_i$, 
		and thus $f(c)\in {A_i}_MA_i=0$. 
		This implies that $C_MA_i=0$, for all $i\in I$, and hence $C_MC=0$. 
		It follows that $C\in \Gamma$, and so, by Zorn's lemma, 
		$\Gamma$ has maximal elements.
	\end{proof}
	
	
	\begin{thm}\label{Knil}
		Let $M$ be a projective generator of $\sm$. 
		If $M$ has Krull dimension,
		then the following conditions hold:
		\begin{enumerate}
			\item the prime radical of $M$ is nilpotent;
			\item $M$ has finite reduced rank in $\sm$.
		\end{enumerate}
	\end{thm}
	
	\begin{proof}
		(1)
		In the proof of \cite[Theorem 2.16]{Pepenil}, 
		it can be seen that the condition that $M$ is finitely generated 
		is only necessary to apply \cite[Lemma 2.14]{Pepenil} 
		in the proof of \textit{ii)}$\Rightarrow$\textit{iii)}. 
		Therefore, by Lemma \ref{ObsPPnil}, 
		we can refer to \cite[Theorem 2.16]{Pepenil} 
		without the ``finitely generated'' condition. 
		This also implies that \cite[Theorem 3.10]{Pepenil} 
		is true for a module that is not finitely generated. 
		By taking $\tau=(0,\sm)$ 
		(the least hereditary torsion theory in $\sm$) 
		in \cite[Theorem 3.10]{Pepenil}, 
		it follows that $\mathfrak{L} (M)$ is nilpotent.
		
		(2)
		Let $\gamma$ be the torsion theory in $\sm$ 
		cogenerated by $E^{[M]}(M/\mathfrak{L}(M))$. 
		Since $M$ has Krull dimension, 
		for all $A\in\Sat_\gamma(M)$ 
		the factor module $M/A$ has finite uniform dimension.
		It follows from \cite[Corollary 2.12]{maugoldie} 
		that $M/\mathfrak{L}(M)$ is a semiprime Goldie module. 
		We then have that $M/\gamma(M)$ has Krull dimension 
		and its prime radical is $\mathfrak{L}(M)/\gamma(M)$, 
		so $\mathfrak{L}(M)/\gamma(M)$ is nilpotent by part~(1) of the theorem. 
		Therefore $\mathfrak{L}(M)^k\subseteq \gamma(M)$ for some $k>0$, 
		and it follows from Theorem \ref{rrt} 
		that $M$ has finite reduced rank in $\sm$.
	\end{proof}
	
	\begin{cor}
		Let $M$ be a projective generator in $\sm$,
		and suppose that $M$ has Krull dimension. 
		Then $M/A$ has finite reduced rank in $\sigma[M/A]$, 
		for every fully invariant submodule $A$ of $M$.
	\end{cor}
	
	\begin{proof}
		Let $A$ be a fully invariant submodule of $M$. 
		Then $M/A$ is a projective generator in $\sigma[M/A]$ by Lemma \ref{genq}. 
		Since $M$ has Krull dimension, so does $M/A$. 
		By Theorem \ref{Knil}~(2), $M/A$ has finite reduced rank in $\sigma[M/A]$.
	\end{proof}
	
	
	%

	\begin{prop}\label{rrq}
		Let $M$ be projective in $\sm$. 
		Suppose that $M$ has finite reduced rank in $\sm$ 
		and $N\leq M$ is a semiprime submodule. 
		Then $M/N$ has finite reduced rank in $\sigma[M/N]$ 
		if and only if $M/N$ is a Goldie module.
	\end{prop}
	
	\begin{proof}
		$\Rightarrow$) 
		Suppose that $M/N$ has finite reduced rank in $\sigma[M/N]$. 
		Since $N$ is a semiprime submodule, 
		the prime radical of $M/N$ is zero, 
		and so $M/N$ is a Goldie module by Theorem \ref{rrt}. 
		
		$\Leftarrow$) 
		By hypothesis $M/N$ is a semiprime Goldie module, 
		and, moreover, $M/N$ is projective in $\sigma[M/N]$. 
		It follows from Theorem \ref{goldiefl} that
		$M/N$ has finite reduced rank in $\sigma[M/N]$. 
	\end{proof}
	

	\begin{prop}
		\label{rrs}
		Let $M$ be a projective in $\sm$. 
		If $M$ has finite reduced rank in $\sm$, 
		then $M^{(\ell)}$ has finite reduced rank in $\sigma[M^{(\ell)}]$, 
		for all positive integers $\ell$.
	\end{prop}
	
	\begin{proof}
		Given a positive integer $\ell$,
		it is clear that $\sm=\sigma[M^{(\ell)}]$. 
		Since $\mathfrak{L}(M^{(\ell)})=\mathfrak{L}(M)^{(\ell)}$, 
		we have that 
		$\gamma=\chi(E^{[M]}(M/\mathfrak{L}(M)))=
		\chi\left( E^{[M]}\left( M^{(\ell)}/\mathfrak{L}(M^{(\ell)})\right) \right) $, 
		and therefore 
		$\mathcal{Q}_\gamma(M^{(\ell)})=\mathcal{Q}_\gamma(M)^{(\ell)}$ 
		has finite length.
	\end{proof}

	%
	%
	
	\begin{lemma}\label{gentt}
		Let $M$ and $G$ be projective generators of $\sm$.  Then 
		\[\chi(E^{[M]}(M/\mathfrak{L}(M)))=\chi(E^{[G]}(G/\mathfrak{L}(G))).\]
	\end{lemma}
	
	\begin{proof}
		Since $M$ and $G$ are projective generators, 
		there exist indexing sets $I$ and $J$ 
		such that $G^{(I)}=M\oplus A$ and $M^{(J)}=G\oplus N$. 
		It follows that 
		$\mathfrak{L}(G)^{(I)}=\mathfrak{L}(G^{(I)})
		=\mathfrak{L}(M)\oplus \mathfrak{L}(A)$. 
		Therefore,
		\[\left( G/\mathfrak{L}(G)\right)^{(I)}\cong G^{(I)}/\mathfrak{L}(G)^{(I)}
		\cong \frac{M\oplus A}{\mathfrak{L}(M)\oplus 
			\mathfrak{L}(A)}=M/\mathfrak{L}(M)\oplus A/\mathfrak{L}(A).\]
		This implies that there is a monomorphism 
		$M/\mathfrak{L}(M)\to E^{[G]}\left( G/\mathfrak{L}(G)\right)^{I}$,
		and hence 
		$\chi(E^{[M]}(M/\mathfrak{L}(M)))\subseteq\chi(E^{[G]}(G/\mathfrak{L}(G)))$. 
		
		The proof that 
		$\chi(E^{[G]}(G/\mathfrak{L}(G))) 
		\subseteq
		\chi(E^{[M]}(M/\mathfrak{L}(M)))$
		is similar.
	\end{proof}
	
	
	For a module $_R M$,
	recall that $\add (M)$ denotes the \emph{additive closure} 
	of $\left\{ M \right\}$ in $R$-Mod,
	consisting of all left $R$-modules isomorphic to a direct summand 
	of a direct sum $M^{(n)}$ of $n$ copies of $M$, for some $n>0$.
	
	\begin{prop}\label{rrpg}
		Let $M$ and $G$ be projective generators of $\sm$ 
		such that $M\in\add(G)$ and $G\in\add(M)$. 
		Then $M$ has finite reduced rank in $\sm$ 
		if and only if 
		$G$ has finite reduced rank in $\sigma[G]$.
	\end{prop}
	
	\begin{proof}
		Suppose that $M$ has finite reduced rank, 
		i.e.\ that $\mathcal{Q}_\gamma(M)$ has finite length 
		in the quotient category $\sm/\gamma$, 
		where $\gamma=\chi(E^{[M]}(M/\mathfrak{L}(M)))$.
		It is clear from the hypothesis that $\sm = \sigma [G]$,
		and Lemma \ref{gentt} shows that $\gamma=\chi(E^{[G]}(G/\mathfrak{L}(G)))$,
		so it suffices to show that $\mathcal{Q}_\gamma(G)$ has finite length.
		Since $G\in\add(M)$, there exists $n>0$ such that $M^{(n)}=G\oplus A$. 
		Therefore 
		$\mathcal{Q}_\gamma(M^{(n)})\cong\mathcal{Q}_\gamma(M)^{(n)}
		=\mathcal{Q}_\gamma(G)\oplus\mathcal{Q}_\gamma(A)$ 
		has finite length, 
		and this implies that $\mathcal{Q}_\gamma(G)$ has finite length. 
	\end{proof}
	
	\begin{prop}\label{rreq}
		Let $M$ be projective in $\sm$, 
		and let $F:\sm\to \sigma[G]$ be an equivalence. 
		If $M$ has finite reduced rank in $\sm$,
		then $F(M)$ has finite reduced rank in $\sigma[F(M)]$.
	\end{prop}
	
	\begin{proof}
		If $M$ is projective in $\sm$,
		then $F(M)$ is projective in $\sigma[G]=\sigma[F(M)]$. 
		Letting $E$ denote the module $E^{[M]}(M/\mathfrak{L}(M))\in\sm$, 
		it follows that $F(E)$ is an injective module in $\sigma[F(M)]$. 
		Let $\overline{\gamma}$ denote 
		the torsion theory in $\sigma[F(M)]$ cogenerated by $F(E)$. 
		A minor modification of the proof of \cite[Corollary 5.2]{beachy2020fully} 
		shows that there exists a bijective correspondence 
		between the prime submodules of $M$ 
		and the prime submodules of $F(M)$. 
		Since $F$ preserves the lattice structure of $M$, 
		we have $F(\mathfrak{L}(M))=\mathfrak{L}(F(M))$. 
		Furthermore, $F$ preserves injective hulls, 
		and hence 
		$F(E^{[M]}(M/\mathfrak{L}(M)))
		\cong E^{[F(M)]}(F(M/\mathfrak{L}(M)))
		=E^{[F(M)]}(F(M)/F(\mathfrak{L}(M)))$. 
		Thus the image of $\gamma$ under $F$ is $\overline{\gamma}$, 
		and then $F$ induces an equivalence 
		between the quotient categories $\sm/\gamma$ 
		and $\sigma[F(M)]/\overline{\gamma}$. 
		We conclude that if $\mathcal{Q}_\gamma(M)$ has finite length, 
		then $\mathcal{Q}_{\overline{\gamma}}(F(M))$ has finite length.
	\end{proof}
	
	The above propositions can be used to provide an alternate proof 
	of the following result
	\cite[Theorem 4]{beachy1982max}.
	
	\begin
	{cor}\label{rrMorita}
	Finite reduced rank is a Morita invariant property for rings.
\end{cor}

\begin{proof}
	Let $R$ and $S$ be rings, 
	and let $F:R\text{-Mod}\to S\text{-Mod}$ be an equivalence. 
	It follows from Proposition \ref{rreq} 
	that $R$ has finite reduced rank 
	if and only if 
	$F(R)$ has finite reduced rank in $\sigma[F(R)]=S\text{-Mod}$. 
	Since $S$ is a projective generator of $\sigma[F(R)]$, 
	Proposition \ref{rrpg} shows that $S$ has finite reduced rank.
\end{proof}

Recall that a finitely generated projective generator 
is called a \emph{progenerator}.

\begin{cor}\label{rrendo}
	Let $M$ be a progenerator of $\sm$. 
	Then $M$ has finite reduced rank in $\sm$ 
	if and only if 
	the ring $S=\End_R(M)$ has finite reduced rank.
\end{cor}

\begin{proof}
	In this case it follows from \cite[46.2]{wisbauerfoundations}
	that $\sm \cong S\text{-Mod}$,
	and the proof proceeds as in Corollary~\ref{rrMorita}.
\end{proof}

\begin{lemma}\label{gends}
	Let $M$ be a generator of $\sm$, and suppose that $M=N\oplus L$. 
	If $\Hom_R(L,N)=0$, then $N$ is a generator of $\sigma[N]$.
\end{lemma}

\begin{proof}
	If $V\in\sigma[N]$, 
	then there exists an epimorphism $\rho:N^{(X)}\to W$ such that $V\leq W$. 
	Consider $\rho^{-1}(V)\subseteq N^{(X)}$. 
	Since $M$ is a generator of $\sm$, 
	there exists an epimorphism $\pi:M^{(Y)}\to \rho^{-1}(V)$. 
	On the other hand, 
	$M^{(Y)}=N^{(Y)}\oplus L^{(Y)}$ 
	and by hypothesis $\Hom_R(L^{(Y)},\rho^{-1}(V))=0$. 
	Hence there is an epimorphism $N^{(X)}\to \rho^{-1}(V)\to V$, 
	showing that $N$ is a generator of $\sigma[N]$.
\end{proof}

\begin{prop}\label{rrds}
	Let $M$ be a projective generator of $\sm$. 
	Suppose that $M$ has finite reduced rank in $\sm$ and 
	$M=N\oplus L$ with $\Hom_R(L,N)=0$. 
	Then $N$ has finite reduced rank in $\sigma[N]$.
\end{prop}

\begin{proof}
	Note that $N$ is projective in $\sigma[N]$ 
	since $M$ is projective in $\sm$. 
	It follows from Lemma \ref{gends} 
	that $N$ is a projective generator of $\sigma[N]$. 
	Let $\gamma$ be the hereditary torsion theory 
	cogenerated by $E^{[M]}(M/\mathcal{L}(M))$ in $\sm$,
	and let $\overline{\gamma}$ be the hereditary torsion theory 
	cogenerated by $E^{[N]}(N/\mathcal{L}(N))$ in $\sigma[N]$. 
	Since $E^{[N]}(N/\mathcal{L}(N))\subseteq E^{[M]}(N/\mathcal{L}(N))$, 
	we have $\gamma\cap\sigma[N]\subseteq\overline{\gamma}$. 
	Let $A\in \Sat_{\overline{\gamma}}(N)$. 
	Then we have
	\[\frac{M}{A\oplus L}=\frac{N\oplus L}{A\oplus L}\cong N/A,\]
	and since $N/A$ is $\overline{\gamma}$ torsionfree, 
	it follows that $\frac{M}{A\oplus L}$ is $\gamma$-torsionfree, 
	so $A\oplus L\in\Sat_\gamma(M)$. 
	This implies that there exists an embedding 
	$\Sat_{\overline{\gamma}}(N)\hookrightarrow\Sat_\gamma(M)$. 
	Therefore $\Sat_{\overline{\gamma}}(N)$ satisfies the ACC,
	and by Theorem \ref{rrt} we conclude that 
	$N$ has finite reduced rank in $\sigma[N]$.
\end{proof}

Recall that an idempotent $e$ of a ring $R$ 
is called \emph{right semicentral} 
if $er=ere$ for all $r\in R$. 
Note that if $e\in R$ is right semicentral, then 
\[eR(1-e)\cong\Hom_R(R(1-e),Re)=0 .\]

\begin{cor}\label{corner}
	If $R$ has finite reduced rank 
	and $e\in R$ is a right semicentral idempotent, 
	then $eRe$ has finite reduced rank.
\end{cor}

\begin{proof}
	It follows from Proposition \ref{rrds} 
	that $Re$ has finite reduced rank in $\sigma[Re]$. 
	Since $Re$ is finitely generated, 
	$\End_R(Re)\cong eRe$ has finite reduced rank by Corollary \ref{rrendo}.
\end{proof}

In general, we do not know whether 
the corner ring of a ring with finite reduced rank has finite reduced rank. 

\section{A generalization of Small's Theorem}
\label{gST}

In this section,
let $_R M$ be a module with a fully invariant submodule $N$,
and let $S$ denote the endomorphism ring of $M$. 
Given an endomorphism $f:M\to M$, 
since $N$ is fully invariant 
it follows that $f$ induces an endomorphism 
$\overline{f}:M/N\to M/N$ such that $\overline{f}\pi=\pi f$, 
where $\pi:M\to M/N$ is the canonical projection. 
If we define $\mathcal{C}(N)$ to be the subset
$\{f\in S\mid \overline{f}\;\text{ is a monomorphism}\}$,
then it is clear that
$\mathcal{C}(N)$ is a multiplicatively closed subset of $S$.

\begin{defn}
	Let $_R M$ be a module, and let $N$ be a fully invariant submodule of $M$. 
	A module $L$ in $\sm$ is called 
	\emph{$\mathcal{C}(N)$-torsion} 
	if for every $f\in\Hom_R(M,L)$ there exists $c\in\mathcal{C}(N)$ 
	such that $fc=0$. 
	That is, $\Hom_R(M,L)$ is a $\mathcal{C}(N)$-torsion right $S$-module.
\end{defn}

Given a module $L\in\sm$, 
let $\tau_N(L)$ be the sum of all 
$\mathcal{C}(N)$-torsion submodules of $L$. 

\begin{lemma}\label{prad}
	Let $M$ be projective in $\sm$, 
	and let $N$ be a fully invariant submodule of $M$. 
	Then $\tau_N$ is a left exact radical.
\end{lemma}

\begin{proof}
	Let $K$ and $L$ be in $\sm$, 
	and let $f:K\to L$ be any homomorphism. 
	Suppose that $K'\leq K$ is a $\mathcal{C}(N)$-torsion submodule, 
	and consider $g\in\Hom_R(M,f(K'))$. 
	Since $M$ is projective in $\sm$, 
	there exists $h:M\to K'$ such that $fh=g$. 
	It follows that there exists $c\in\mathcal{C}(N)$ such that $hc=0$, 
	and therefore $gc=fhc=0$. 
	Thus $f(K')$ is $\mathcal{C}(N)$-torsion,
	and hence $f(\tau_N(K))\subseteq \tau_N(L)$. 
	
	Now suppose that $K/\tau_N(L)$ 
	is a $\mathcal{C}(N)$-torsion submodule of $L/\tau_N(L)$, 
	and let $f:M\to K$ be any homomorphism. 
	Then there exists $c\in\mathcal{C}(N)$ such that $\pi fc=0$,
	where $\pi:L\to L/\tau_N(L)$ is the canonical projection. 
	It follows that $fc(M)\subseteq \tau_N(L)$, 
	and so there exists $d\in\mathcal{C}(N)$ such that $fcd=0$. 
	This implies that $K$ is $\mathcal{C}(N)$-torsion,
	and therefore $K/\tau_N(L)=0$. 
	
	It is clear that if $K\leq L$, then $\tau_N(K)=\tau_N(L)\cap K$, 
	showing that $\tau_N$ is left exact.
\end{proof}

\begin{rem}
	By Lemma \ref{prad}, 
	for a fully invariant submodule $N\leq M$ 
	the radical $\tau_N$ defines a hereditary torsion theory in $\sm$, 
	which will be also denoted $\tau_N$.
\end{rem}

\begin{prop}\label{ctorgamma}
	Let $M$ be a projective generator of $\sm$ 
	with prime radical $\mathfrak{L}(M)$. 
	Let $\gamma$ be the hereditary torsion theory in $\sm$ 
	cogenerated by $E^{[M]}(M/\mathfrak{L}(M))$, 
	and let $\tau_{\mathfrak{L}(M)}$ be the hereditary torsion theory 
	defined by the $\mathcal{C}(\mathfrak{L}(M))$-torsion modules in $\sm$. 
	If $M$ has finite reduced rank in $\sm$, 
	then $\gamma=\tau_{\mathfrak{L}(M)}$.
\end{prop}

\begin{proof}
	Let $E$ denote the $M$-injective module $E^{[M]}(M/\mathfrak{L}(M))$, 
	and let $f:M\to M/\mathfrak{L}(M)$ be any homomorphism. 
	Then there exists $g:M\to M$ such that $\pi g=f$,
	where $\pi:M\to M/\mathfrak{L}(M)$ is the canonical projection. 
	If there exists $c\in\mathcal{C}(\mathfrak{L}(M))$ such that $fc=0$, 
	then $0=fc=\pi gc$, 
	and hence $gc(M)\subseteq\mathfrak{L}(M)$. 
	It follows that $\overline{g}\overline{c}=0$ 
	for $\overline{c},\overline{g}:M/\mathfrak{L}(M)\to M/\mathfrak{L}(M)$. 
	Since $M/\mathfrak{L}(M)$ has finite uniform dimension by Theorem \ref{rrt}, 
	we have $\overline{c}(M/\mathfrak{L}(M))\ess M/\mathfrak{L}(M)$. 
	This implies that $\Ker\overline{g}\ess M/\mathfrak{L}(M)$,
	and so $\Img\overline{g}$ is $\frac{M}{\mathfrak{L}(M)}$-singular. 
	Then $\Img\overline{g}=0$ 
	because $M/\mathfrak{L}(M)$ is a semiprime Goldie module, 
	and therefore $g(M)\subseteq\mathfrak{L}(M)$, 
	so $0=\pi g=f$. 
	That is, $M/\mathfrak{L}(M)$ is $\mathcal{C}(\mathfrak{L}(M))$-torsionfree, 
	showing that $\tau_{\mathfrak{L}(M)}\leq\gamma$.
	
	Now suppose that $\tau_{\mathfrak{L}(M)}<\gamma$. 
	Then there exists $L \neq 0$ such that $L$ is $\gamma$-torsion 
	but $\tau_{\mathfrak{L}(M)}(L)=0$. 
	Let $f:M\to L$. 
	Since $L$ is $\gamma$-torsion, so is $M/\Ker f$, 
	and hence $\frac{M}{\Ker f+\mathfrak{L}(M)}$ is $\gamma$-torsion. 
	Recall that $M/\mathfrak{L}(M)$ is a semiprime Goldie module 
	and that $\gamma\cap\sigma[M/\mathfrak{L}(M)]$ 
	coincides with $\tau_g$ (see Section \ref{qcsgm}). 
	Then $\frac{M}{\Ker f+\mathfrak{L}(M)}$ is $M/\mathfrak{L}(M)$-singular, 
	and since $M/\mathfrak{L}(M)$ is non-$M/\mathfrak{L}(M)$-singular, 
	the submodule $\frac{\Ker f+\mathfrak{L}(M)}{\mathfrak{L}(M)}$ 
	is essential in $M/\mathfrak{L}(M)$. 
	It follows from Theorem \ref{rrt} and \cite[Theorem 2.8]{maugoldie}
	that there exists a monomorphism 
	$\overline{c}:M/\mathfrak{L}(M)\to 
	\frac{\Ker f+\mathfrak{L}(M)}{\mathfrak{L}(M)}$. 
	Since $M$ is projective in $\sm$, 
	there exists a lifting $c:M\to \Ker f$ of $\overline{c}$,
	and then $c\in\mathcal{C}(\mathfrak{L}(M))$ and $fc=0$. 
	That is, $\tau_{\mathfrak{L}(M)}(L)\neq 0$, 
	which implies that $f$ must be zero. 
	Since $M$ is a generator of $\sm$, it follows that $L=0$, 
	contradicting the choice of $L$, and thus $\tau_{\mathfrak{L}(M)}=\gamma$.
\end{proof}

\begin{prop}\label{55}
	Let $M$ be a projective generator of $\sm$ 
	with prime radical $\mathfrak{L}(M)$. 
	Consider the hereditary torsion theory $\tau_{\mathfrak{L}(M)}$ in $\sm$. 
	The following conditions are equivalent:
	\begin{enumerate}[label=\emph{(\alph*)}]
		\item $\mathcal{C}(\mathfrak{L}(M))$ is a right Ore set;
		\item $M/c(M)$ is $\tau_{\mathfrak{L}(M)}$-torsion for all 
		$c\in\mathcal{C}(\mathfrak{L}(M))$.
	\end{enumerate} 
\end{prop}

\begin{proof}	
	(a)$\Rightarrow$(b) 
	Let $c\in\mathcal{C}(\mathfrak{L}(M))$, 
	let $f:M\to M/c(M)$ be any homomorphism, 
	and consider the canonical projection $\pi:M\to M/c(M)$. 
	Then there exists $g:M\to M$ such that $\pi g=f$. 
	From the right Ore condition, 
	there exist $h:M\to M$ and $d\in\mathcal{C}(\mathfrak{L}(M))$ 
	such that $gd=ch$. 
	Therefore, $fd=\pi gd=\pi ch=0$, 
	and thus $M/c(M)$ is $\tau_{\mathfrak{L}(M)}$-torsion.
	
	(b)$\Rightarrow$(a) 
	Let $f:M\to M$ and $c\in\mathcal{C}(\mathfrak{L}(M))$. 
	If $\pi:M\to M/c(M)$ is the canonical projection, 
	then by hypothesis there exists $d\in\mathcal{C}(\mathfrak{L}(M))$ 
	such that $\pi fd=0$,
	and this implies that $fd(M)\subseteq c(M)$. 
	Since $M$ is projective in $\sm$, 
	there exists $h:M\to M$ such that $ch=fd$, 
	showing that $\mathcal{C}(\mathfrak{L}(M))$ is a right Ore set.
\end{proof}

\begin{lemma}\label{mudimtf}
	Let $M$ be a projective generator in $\sm$,
	and let $\gamma$ be the hereditary torsion theory in $\sm$ 
	cogenerated by $E^{[M]}(M/\mathfrak{L}(M))$. 
	Suppose that $\mathcal{C}(\mathfrak{L}(M))\subseteq\mathcal{C}(0)$. 
	If $M$ has finite reduced rank in $\sm$, then 
	\begin{enumerate}
		\item $M$ has finite uniform dimension;
		\item $M$ is $\gamma$-torsionfree;
		\item $M/c(M)$ is $\gamma$-torsion for all 
		$c\in\mathcal{C}(\mathfrak{L}(M))$.
	\end{enumerate}
\end{lemma}

\begin{proof}
	Let $c\in\mathcal{C}(\mathfrak{L}(M))$. 
	Since $\mathcal{C}(\mathfrak{L}(M))\subseteq\mathcal{C}(0)$, 
	it follows that $c$ is a monomorphism 
	and $c(M)$ is $\gamma$-torsionfree (Proposition \ref{ctorgamma}). 
	Therefore $c(M)\cap\gamma(M)=0$, 
	where $\gamma(M)$ is the $\gamma$-torsion submodule of $M$. 
	This implies that there exists an embedding $c(M)\to M/\gamma(M)$. 
	It follows from Theorem \ref{rrt} 
	that $M\cong c(M)$ has finite uniform dimension. 
	Since $c$ is a monomorphism, 
	we have $c(M)\ess M$, and hence $\gamma(M)=0$, 
	showing that $M$ is $\gamma$-torsionfree. 
	Let $c\in\mathcal{C}(\mathfrak{L}(M))$,
	and consider the exact sequence
	$0\to c(M)\overset{i}{\rightarrow} M\to M/c(M)\to 0$. 
	By hypothesis $\mathcal{Q}_\gamma(M)$ has finite length, 
	and therefore $\mathcal{Q}_\gamma(c(M))$ also has finite length,
	since $c(M)\cong M$. 
	This implies that $\mathcal{Q}_\gamma(i)$ is an isomorphism,
	and so $M/c(M)$ is $\gamma$-torsion by \cite[9.18]{wisbauermodules}.
\end{proof}

\begin{thm}\label{smallgen}
	Let $M$ be a projective generator of $\sm$ with $S=\End_R(M)$, 
	and let $\gamma$ be the hereditary torsion theory in $\sm$ 
	cogenerated by $E^{[M]}(M/\mathfrak{L}(M))$. 
	The following conditions are equivalent:
	\begin{enumerate}[label=\emph{(\alph*)}]
		\item $M$ has finite reduced rank in $\sm$ and 
		$\mathcal{C}(\mathfrak{L}(M))\subseteq\mathcal{C}(0)$; 
		\item $T=\End_R(\mathcal{Q}_\gamma(M))$ is Artinian and 
		$S$ is a right order in $T$.
	\end{enumerate}
\end{thm}

\begin{proof}
	(a)$\Rightarrow$(b) 
	Consider $\mathcal{Q}_\gamma(M)$ as an $R$-module. 
	We claim that $S$ is an order in $T=\End_R(\mathcal{Q}_\gamma(M))$. 
	It follows from Lemma \ref{mudimtf} 
	that $\varphi_M:M\to \mathcal{Q}_\gamma(M)$ is a monomorphims. Then, every $f\in S$, has a unique extension $\overline{f}\in T$, by definition of the module of quotients. 
	Hence we can consider $S$ to be a subring of $T$. 
	
	Let $q\in T$ and consider $A=q^{-1}(M)\cap M$. 
	Then $M/A\cong\frac{q^{-1}(M)+M}{q^{-1}(M)}$,
	so $M/A$  can be embedded in $\mathcal{Q}_\gamma(M)/M$,
	which implies that $M/A$ is $\gamma$-torsion. 
	Let $\pi:M\to M/A$ be the canonical projection. 
	Then there exists $c\in\mathcal{C}(\mathfrak{L}(M))$ such that $\pi c=0$, 
	and this implies that $c(M)\subseteq A$. 
	By Lemma \ref{mudimtf}, $c\in T$ is invertible, 
	and then $q=(qc)c^{-1}=tc^{-1}$. 
	This shows that $S$ is a right order in $T$. 
	
	To see that $T$ is an Artinian ring, 
	note that $\mathcal{Q}_\gamma(M)$ 
	is a projective generator of finite length in $\sm/\gamma$. 
	This implies that $\Hom_{\sm/\gamma}(\mathcal{Q}_\gamma(M),\_)$ 
	gives an equivalence between $\sm/\gamma$ 
	and the category Mod-$T$ \cite[Ch. X, Theorem 4.1]{stenstromrings}.  
	Therefore 
	$T=\End_R(\mathcal{Q}_\gamma(M))=\End_{\sm/\gamma}(\mathcal{Q}_\gamma(M))$ 
	is an Artinian ring.
	
	(b)$\Rightarrow$(a) 
	Since $M$ is a projective generator in $\sm$, 
	the localization $\mathcal{Q}_\gamma(M)$ 
	is a projective generator in $\sm/\gamma$. 
	Hence there is a quotient category Mod-$T/\tau$ 
	such that the functor 
	$\Hom_{\sm/\gamma}(\mathcal{Q}_\gamma(M),\_)$ 
	gives an equivalence \cite[Ch. X, Theorem 4.1]{stenstromrings},
	as shown in the following diagram.
	
	\[\xymatrix{ & \text{Mod-}T/\tau\ar@{_(-{>}}[rd]^i & \\ 
		\sm/\gamma\ar[rr]^-{\Hom_{\sm/\gamma}(\mathcal{Q}_\gamma(M),\_)}
		\ar[ur]_\cong^{\mathcal{H}} & & \text{Mod-}T \\ 
		\sm\ar[u]^{\mathcal{Q}_\gamma}
		& & \text{Mod-}S\ar[u]_{\mathcal{F}} }\]
	
	We obtain the functor $\mathcal{F}$ 
	from the fact that $S$ is a right order in $T$. 
	The functor $\mathcal{H}$ 
	is the corestriction of the functor 
	$\Hom_{\sm/\gamma}(\mathcal{Q}_\gamma(M),\_)$,
	which is an equivalence,
	and $\mathcal{H}(\mathcal{Q}_\gamma(M))=T$. 
	It follows that $\mathcal{Q}_\gamma(M)$ is Artinian 
	because $\mathcal{H}$ is an equivalence, 
	and this proves that $M$ has finite reduced rank.
	
	It remains to prove that 
	$\mathcal{C}(\mathfrak{L}(M))\subseteq\mathcal{C}(0)$. 
	We first show that $M$ is $\gamma$-torsionfree. 
	Note that the inclusion  of $S$ in $T$ is constructed as follows: 
	given $\varphi\in S$, 
	there exists $\widehat{\varphi}:M/\gamma(M)\to M/\gamma(M)$ 
	such that $\pi\widehat{\varphi}=\widehat{\varphi}\pi$,
	and by the  definition of $\mathcal{Q}_\gamma(M)$, 
	the mapping $\widehat{\varphi}$ can be uniquely extended 
	to a endomorphism $\overline{\varphi}\in T$.
	This leads to the following diagram.
	
	\[\xymatrix{M\ar@{-{>>}}[r]^-\pi \ar[d]_\varphi 
		& M/\gamma(M)\ar[r]\ar[d]^{\widehat{\varphi}} 
		& \mathcal{Q}_\gamma(M)\ar@{--{>}}[d]^{\overline{\varphi}} \\ 
		M\ar@{-{>>}}[r]_-\pi 
		& M/\gamma(M)\ar[r] 
		& \mathcal{Q}_\gamma(M)}\]
	
	\noindent
	Suppose that $\gamma(M)\neq 0$. 
	Since $M$ is a generator of $\sm$, 
	there exists $0\neq\varphi:M\to \gamma(M)$, 
	and hence $\overline{\varphi}=0\in T$. 
	This implies that $\varphi=0$, because the assignment
	$\varphi\mapsto\overline{\varphi}$ is injective. 
	This is a contradiction, 
	showing that $\gamma(M)=0$.
	
	Note that there is a surjective ring homomorphism 
	$\Theta:S\to \End_R(M/\mathfrak{L}(M))$ with 
	$\Ker\Theta=\Hom_R(M,\mathfrak{L}(M))$. 
	By Theorem \ref{rrt}, 
	$M/\mathfrak{L}(M)$ is a semiprime Goldie module,
	and hence $\End_R(M/\mathfrak{L}(M))$ 
	is a semiprime right Goldie ring \cite[Theorem 2.22]{maugoldie}. 
	This implies that $\Ker\Theta$ is a semiprime ideal of $S$, 
	and therefore $\mathfrak{L}(S)\subseteq\Ker\Theta$. 
	On the other hand, 
	$\mathfrak{L}(M)$ is nilpotent 
	since $M$ has finite reduced rank and $M$ is $\gamma$-torsionfree 
	(Theorem \ref{rrt}). 
	It follows that $\Hom_R(M,\mathfrak{L}(M))$ is a nilpotent ideal of $S$, 
	and thus $\Ker\Theta=\Hom_R(M,\mathfrak{L}(M))\subseteq\mathfrak{L}(S)$, 
	so $\Ker\Theta=\mathfrak{L}(S)$. 
	Now let $c\in\mathcal{C}(\mathfrak{L}(M))$. 
	Then $\Theta(c)\in\End_R(M/\mathfrak{L}(M))$ is injective. Since $M/\mathfrak{L}(M)$ is a semiprime Goldie module, projective in $\sigma\left[ M/\mathfrak{L}(M)\right]$, $\Theta(c)$ is a regular element \cite[11.5]{wisbauermodules}. 
	That is, $c$ is a regular element modulo $\mathfrak{L}(S)$, 
	and then $c$ is regular in $S$ 
	since $S$ is a right order in an Artinian ring. 
	Suppose that $\Ker c\neq 0$. 
	Since $M$ is a generator, 
	there exists $0\neq f\in\Hom_R(M,\Ker c)$, 
	which implies that $cf=0$. 
	Then $f=0$ because $c$ is regular, 
	and thus $\Ker c=0$, 
	proving that $\mathcal{C}(\mathfrak{L}(M))\subseteq\mathcal{C}(0)$.
\end{proof}

\begin{cor}
	Let $R$ be a ring, $e\in R$ a semicentral idempotent. If $R$ is an order in an Artinian ring, then so are the rings $\Mat_n(R)$ and $eRe$.
\end{cor}

\begin{proof}
	Suppose $R$ is a left order in an Artinian ring. Then $R$ has finite reduced rank on the left and $\mathcal{C}(\mathfrak{L}(R))\subseteq\mathcal{C}(0)$ by Theorem \ref{smallgen}. It follows from Proposition \ref{rrs} that $R^{(n)}$ has finite reduced rank for every $n>0$. On the other hand $\mathfrak{L}(R^{(n)})=\mathfrak{L}(R)^{(n)}$. Hence $\mathcal{C}(\mathfrak{L}(R^{(n)}))\subseteq\mathcal{C}(0)$. Thus $\End_R(R^{(n)})\cong\Mat_n(R)$ is a left order in an Artinian ring by Theorem \ref{smallgen}.
	
	Now, let $e^2=e\in R$ be a semicentral idempotent. By Corollary \ref{corner}, $Re$ has finite reduced rank. We have that $\mathfrak{L}(R)=\mathfrak{L}(Re)\oplus\mathfrak{L}(R(1/e))$. Let $c\in\mathfrak{L}(Re)$, that is, $\bar{c}:Re/\mathfrak{L}(Re)\to Re/\mathfrak{L}(Re)$ is a monomorphism. Thus 
	\[\bar{c}\oplus\bar{1}:\frac{Re}{\mathfrak{L}(Re)}\oplus \frac{R(1-e)}{\mathfrak{L}(R(1-e))}\rightarrow \frac{Re}{\mathfrak{L}(Re)}\oplus \frac{R(1-e)}{\mathfrak{L}(R(1-e))}\]
	is a monomorphism and $\frac{Re}{\mathfrak{L}(Re)}\oplus \frac{R(1-e)}{\mathfrak{L}(R(1-e))}=R/\mathfrak{L}(R)$. By Theorem \ref{smallgen}, $c\oplus 1$ is a monomorphism which implies that so is $c$. Thus, $\mathcal{C}(\mathfrak{L}(Re))\subseteq\mathcal{C}(0)$. By Theorem \ref{smallgen}, $\End_R(Re)=eRe$ is a left order in an Artinian ring. 
\end{proof}

Recall that an $R$-module $M$ is said to be \emph{Rickart} if $\Ker f$ is a direct summand of $M$ for all $f\in\End_R(M)$ \cite{leerickart}. 

\begin{prop}\label{ricgold}
	Let $M$ be a projective generator of $\sm$ with $S=\End_R(M)$, 
	and let $\gamma$ be the hereditary torsion theory in $\sm$ 
	cogenerated by $E^{[M]}(M/\mathfrak{L}(M))$. 
	If $M$ is a Rickart module with Krull dimension, then $T=\End_R(\mathcal{Q}_\gamma(M))$ is Artinian and $S$ is a right order in $T$.
\end{prop}

	\begin{proof}
		By Theorems \ref{Knil} and \ref{smallgen}, we just have to see that $\mathcal{C}(\mathfrak{L}(M))\subseteq\mathcal{C}(0)$. Let $c\in\mathcal{C}(\mathfrak{L}(M))$. This implies that $\Ker c\subseteq\mathfrak{L}(M)$. Since $M$ is Rickart, $\Ker c$ is a direct summand of $M$ and therefore ${\Ker c}_M\Ker c=\Ker c$. By \cite[Corollary 5.5]{mauprimerad}, $\mathfrak{L}(M)$ is nilpotent. Thus $\Ker c=0$. 
	\end{proof}

\begin{cor}
	Let $R$ be a left Rickart ring with Krull dimension, then $R$ is a left order in an Artinian ring.
\end{cor}

\bibliographystyle{acm}
\bibliography{biblio}

\end{document}